\documentclass[12pt,a4paper]{amsart}%

\usepackage{amsfonts,amsmath,amssymb}
\usepackage{amssymb,amsthm,amsxtra}
\usepackage[usenames]{color}
\usepackage{amscd}
\usepackage{amsthm}
\usepackage{amsfonts}
\usepackage{amssymb}
\usepackage{amsmath}
\usepackage{graphicx}
\usepackage{hyperref}
\usepackage{enumerate}%
\usepackage[all]{xy}
\usepackage[usenames,dvipsnames]{xcolor}
\usepackage{mathrsfs}
\usepackage[capitalize]{cleveref}
\usepackage{cleveref}
 \newtheorem*{corollary*}{Corollary}
 \newtheorem*{construction*}{Construction}
 \newtheorem*{definition*}{Definition}
 \newtheorem*{notation*}{Notation}
 \newtheorem*{lemma*}{Lemma}
 \newtheorem*{theorem*}{Theorem}
 \newtheorem*{remark*}{Remark}
 \newtheorem*{example*}{Example}
 \newtheorem*{conjecture*}{Conjecture}
 \newtheorem*{condition*}{Condition}
 \newtheorem*{result*}{Result}
 \newtheorem*{property*}{Property}

 \newtheorem*{cor*}{Corollary}
 \newtheorem*{const*}{Construction}
 \newtheorem*{defn*}{Definition}
 \newtheorem*{notn*}{Notation}
 \newtheorem*{lem*}{Lemma}
 \newtheorem*{thm*}{Theorem}
 \newtheorem*{rem*}{Remark}
 \newtheorem*{exm*}{Example}
 \newtheorem*{conj*}{Conjecture}

 \newtheorem{lemma}{Lemma}[subsection]
 \newtheorem{proposition}[lemma]{Proposition}
 \newtheorem{remark}[lemma]{Remark}
 \newtheorem{example}[lemma]{Example}
 \newtheorem{theorem}[lemma]{Theorem}

 \newtheorem{thm}[lemma]{Theorem}
 \newtheorem{prop}[lemma]{Proposition}
 \newtheorem{lem}[lemma]{Lemma}
 \newtheorem{defn}[lemma]{Definition}
 
 \newtheorem{cor}[lemma]{Corollary}
 \newtheorem{exm}[lemma]{Example}

 \newtheorem{rem}[lemma]{Remark}

 \newtheorem{introtheorem}{Theorem}

 \crefname{introtheorem}{theorem}{theorems}
 \Crefname{introtheorem}{Theorem}{Theorems}
  \newtheorem{introthm}[introtheorem]{Theorem}
   \crefname{introthm}{theorem}{theorems}
 \Crefname{introthm}{Theorem}{Theorems}

  \crefname{introcorollary}{corollary}{corollaries}
 \Crefname{introcorollary}{Corollary}{Corollaries}
 \newtheorem{introcor}[introtheorem]{Corollary}
   \crefname{introcor}{corollary}{corollaries}
 \Crefname{introcor}{Corollary}{Corollaries}
 
   \crefname{introconjecture}{conjectures}{conjectures}
 \Crefname{introconjecture}{Conjecture}{Conjectures}
 
    \crefname{introconj}{conjectures}{conjectures}
 \Crefname{introconj}{Conjecture}{Conjectures}

     \crefname{introlem}{lemma}{lemmas}
 \Crefname{introlem}{Lemma}{Lemmas}
 
 \crefname{introremark}{remark}{remarks}
 \Crefname{introremark}{Remark}{Remarks}
 
  \crefname{introrem}{remark}{remarks}
 \Crefname{introrem}{Remark}{Remarks}

 \newtheorem{introprop}[introtheorem]{Proposition}
   \crefname{introprop}{Proposition}{Propositions}
 \Crefname{introprop}{Proposition}{Propositions}
 \newtheorem{introdefn}[introtheorem]{Definition}
   \crefname{introdefn}{definition}{definitions}
 \Crefname{introdefn}{Definition}{Definitions}
 
   \crefname{intronotn}{notation}{notations}
 \Crefname{intronotn}{Notation}{Notations}
 
   \crefname{introtask}{task}{tasks}
 \Crefname{introtask}{Task}{Tasks}
 
  \crefname{introprob}{problem}{problems}
 \Crefname{introprob}{Problem}{Problems}
 
   \crefname{introquestion}{question}{questions}
 \Crefname{introquestion}{Question}{Questions}
 \newtheorem{introexm}[introtheorem]{Example}
   \crefname{introexm}{example}{example}
 \Crefname{introquestion}{Example}{Example}
 \crefname{theorem}{theorem}{theorems}
 \Crefname{theorem}{Theorem}{Theorems}
  \crefname{thm}{theorem}{theorems}
 \Crefname{thm}{Theorem}{Theorems}
  \crefname{corollary}{Corollary}{Corollaries}
 \Crefname{corollary}{Corollary}{Corollaries}
   \crefname{cor}{Corollary}{Corollaries}
 \Crefname{cor}{Corollary}{Corollaries}
   \crefname{conjecture}{conjectures}{conjectures}
 \Crefname{conjecture}{Conjecture}{Conjectures}

    \crefname{conj}{conjectures}{conjectures}
 \Crefname{conj}{Conjecture}{Conjectures}

     \crefname{lem}{lemma}{lemmas}
 \Crefname{lem}{Lemma}{Lemmas}

      \crefname{lemma}{Lemma}{Lemmas}
 \Crefname{lemma}{Lemma}{Lemmas}

 \crefname{remark}{remark}{remarks}
 \Crefname{remark}{Remark}{Remarks}

  \crefname{rem}{remark}{remarks}
 \Crefname{rem}{Remark}{Remarks}

   \crefname{rem}{remark}{remarks}
 \Crefname{rem}{Remark}{Remarks}

   \crefname{proposition}{Proposition}{Proposition}
 \Crefname{proposition}{Proposition}{Proposition}

    \crefname{prop}{Proposition}{Propositions}
 \Crefname{prop}{Proposition}{Propositions}

   \crefname{defn}{definition}{definitions}
 \Crefname{defn}{Definition}{Definitions}

   \crefname{notn}{notation}{notations}
 \Crefname{notn}{Notation}{Notations}

   \crefname{task}{task}{tasks}
 \Crefname{task}{Task}{Tasks}

  \crefname{prob}{problem}{problems}
 \Crefname{prob}{Problem}{Problems}

   \crefname{question}{question}{questions}
 \Crefname{question}{Question}{Questions}

\newcommand{\alp}{\alpha}

\newcommand{\lam}{\lambda}

\newcommand{\Ind}{\operatorname{Ind}}
\newcommand{\ind}{\operatorname{ind}}
\newcommand{\Ext}{\operatorname{Ext}}

\newcommand{\Id}{\operatorname{Id}}

\renewcommand{\Im}{\operatorname{Im}}
\newcommand{\Ker}{\operatorname{Ker}}

\newcommand{\Tor}{{\operatorname{Tor}}}

\newcommand{\Span}{{\operatorname{Span}}}

\newcommand{\GL}{\operatorname{GL}}
\newcommand{\SL}{\operatorname{SL}}
\newcommand{\Sp}{\operatorname{Sp}}

\newcommand{\oH}{\operatorname{H}}

\newcommand{\Mp}{\operatorname{Mp}}

\newcommand{\ad}{\operatorname{ad}}

\newcommand{\de}{\operatorname{def}}

\newcommand{\Hom}{\operatorname{Hom}}

\newcommand{\An}{\operatorname{An}}
\newcommand{\Ann}{\operatorname{An}}
\newcommand{\AnV}{\mathrm{An}\cV}

\newcommand{\C}{\mathbb{C}}
\newcommand{\bN}{\mathbb{N}}

\newcommand{\bfG}{\mathbf{G}}
\newcommand{\bfH}{\mathbf{H}}

\newcommand{\bfY}{\mathbf{Y}}
\newcommand{\bfX}{\mathbf{X}}
\newcommand{\bfO}{\mathbf{O}}

\newcommand{\bfS}{\mathbf{S}}
\newcommand{\bfR}{\mathbf{R}}
\newcommand{\bfQ}{\mathbf{Q}}
\newcommand{\bfP}{\mathbf{P}}
\newcommand{\bfU}{\mathbf{U}}

\newcommand{\R}{\mathbb{R}}

\newcommand{\Sc}{\cS}

\newcommand{\Rep}{\operatorname{Rep}}

\newcommand{\Fre}{Fr\'echet}

\newcommand{\onto}{\twoheadrightarrow}
\newcommand{\into}{\hookrightarrow}

\providecommand{\fb}{\mathfrak{b}}
\providecommand{\fc}{\mathfrak{c}}

\providecommand{\fg}{\mathfrak{g}}
\providecommand{\fh}{\mathfrak{h}}

\providecommand{\fk}{\mathfrak{k}}
\providecommand{\fl}{\mathfrak{l}}
\providecommand{\fm}{\mathfrak{m}}

\providecommand{\fp}{\mathfrak{p}}
\providecommand{\fq}{\mathfrak{q}}
\providecommand{\fr}{\mathfrak{r}}

\providecommand{\fu}{\mathfrak{u}}

\providecommand{\fz}{\mathfrak{z}}

\providecommand{\cD}{\mathcal{D}}
\providecommand{\cE}{\mathcal{E}}

\providecommand{\cK}{\mathcal{K}}

\providecommand{\cM}{\mathcal{M}}
\providecommand{\cN}{\mathcal{N}}
\providecommand{\cO}{\mathcal{O}}

\providecommand{\cS}{\mathcal{S}}

\providecommand{\cU}{\mathcal{U}}
\providecommand{\cV}{\mathcal{V}}

\providecommand{\Mat}{\mathrm{Mat}}
\providecommand{\fin}{\sigma}

\providecommand{\sub}{\subset}
\providecommand{\susbet}{\subset}

\providecommand{\g}{\mathfrak{g}}

\newcommand{\DimaA}[1]{{{#1}}}
\newcommand{\DimaB}[1]{{{#1}}}
\newcommand{\DimaC}[1]{{{#1}}}
\newcommand{\DimaD}[1]{{{#1}}}
\newcommand{\DimaE}[1]{{{#1}}}
\newcommand{\DimaF}[1]{{{#1}}}
\newcommand{\DimaG}[1]{{{#1}}}
\newcommand{\DimaH}[1]{{{#1}}}
\newcommand{\DimaI}[1]{{{#1}}}

\begin{document}

\title{Finite multiplicities beyond spherical spaces}
\author{Avraham Aizenbud}
\address{Avraham Aizenbud, Faculty of Mathematics and Computer Science, Weizmann
Institute of Science, POB 26, Rehovot 76100, Israel }
\email{aizenr@gmail.com}
\urladdr{http://www.aizenbud.org}

\author{Dmitry Gourevitch}
\address{Dmitry Gourevitch, Faculty of Mathematics and Computer Science, Weizmann
Institute of Science, POB 26, Rehovot 76100, Israel }
\email{dimagur@weizmann.ac.il}
\urladdr{http://www.wisdom.weizmann.ac.il/~dimagur}


\keywords{Representation, algebraic group, nilpotent orbit, spherical space, wave-front set, associated variety of the annihilator, non-commutative harmonic analysis, branching, invariant distribution, Schwartz space, holonomic D-module.}
\subjclass[2010]{20G05, 14L30, 22E46, 22E47, 22E45}
%
%
%
%
%
%
%
%
\date{\today}

\maketitle
\begin{center}
\textbf{with  appendix B by Ido Karshon}
\end{center}

\begin{abstract}
Let $G$ be a real reductive algebraic group, and  let $H\sub G$ be an algebraic subgroup. It is known that the action of $G$ on the space of functions on $G/H$ is "tame" if this space is spherical. In particular, the multiplicities of the space $\Sc(G/H)$ of Schwartz functions on $G/H$ are finite in this case.
In this paper we formulate and analyze a generalization of sphericity that implies finite multiplicities in $\Sc(G/H)$ for small enough irreducible  representations of $G$.
\end{abstract}
%
%


\section{Introduction}

Let $\bf G$ be a reductive algebraic group defined over $\R$, and $\bf X$ be a smooth algebraic $\bf G$-variety. Let $\fg$ denote the Lie algebra of $\bf G$, $\fg^*$ denote the dual space, and $\cN(\fg^*)\sub \fg^*$ denote the nilpotent cone.
Let $\mu:T^*\bfX\to \fg^*$ denote the moment map (see \S \ref{subsec:GeoPrel} for its definition).
\begin{introdefn}\label{def:Ospher}
\begin{enumerate}[(i)]
\item For a nilpotent orbit $\bfO\subset \cN(\fg^*)$ we say that $\bfX$ is $\bfO$-\emph{spherical} if $$\dim \mu^{-1}(\bfO)\leq \dim \bfX + \frac{1}{2}\dim \bfO.$$
\item For a  $\bfG$-invariant subset $\Xi\sub \cN(\fg^*)$, we say that $\bfX$ is $\Xi$-\emph{spherical} if it is $\bfO$-spherical for every orbit $\bfO\subset \Xi$.
\end{enumerate}
\end{introdefn}
\DimaI{A homogeneous $\bfG$-variety $\bf G/H$ is $\bfO$-spherical if and only if $\dim \bfO\cap \fh^{\bot}\leq \dim \bfO/2$, where $\fh^{\bot}\sub \fg^*$ denotes the space of functionals vanishing on the Lie algebra $\fh$ of $\bfH$ (see Corollary \ref{cor:invO} below). 

We prove} the following criterion for sphericity with respect to closures of Richardson orbits, {\it i.e.} orbits that intersect parabolic nilradicals by open dense subsets.

\begin{introthm}[\S \ref{subsec:crit}]\label{IntThm:Ri}
Let $\bfP\subset \bfG$ be a parabolic subgroup, let $\bfO_\bfP$  denote the corresponding Richardson orbit, and $\overline{\bfO_\bfP}$ denote its closure.
Then $\bfX$ is $\overline{\bfO_\bfP}$-spherical if and only if $\bfP$
has finitely many orbits on $\bfX$.
\end{introthm}
This implies in particular  that $\bf X$ is $\{0\}$-spherical if and only if $\bf G$ has finitely many orbits on $\bfX$ and that $\bfX$ is $\cN(\fg^*)$-spherical if and only if $\bfX$ is spherical.

\DimaF{As a byproduct of the proof of Theorem \ref{IntThm:Ri}, we obtained the following theorem.
\begin{introthm}[See Corollary \ref{cor:Rich} below]\label{IntThm:Rich}
Let $\bf P\sub G$ be a parabolic subgroup, and let $\bfO\subset \cN(\fg^*)$ be any Richardson nilpotent orbit. Then the intersection $\bfO\cap \fp^{\bot}$ is  a Lagrangian subvariety of $\bfO$, with respect to the Kirillov-Kostant-Souriau symplectic form on $\bfO$ \DimaC{(see Definition \ref{def:sympO} below).}
\end{introthm}
This theorem is well known in the special case $\bf O=O_P$.
\DimaG{Note that for $\bf O\neq O_P$ the intersection $\bf O\cap \fp^{\bot}$ may be empty.}}

%

Let $G:=\bfG(\R)$ denote the group of real points of $\bfG$ and similarly $X:=\bfX(\R)$. Let $\cM(G)$ denote the category of finitely-generated smooth admissible \Fre $\,$representations of moderate growth (see \cite[\S 11.5]{Wal}).

Our main motivation for the notion of $\bfO$-spherical variety is the following theorem.
\begin{introthm}[See Corollary \ref{cor:fin} below]\label{IntThm:fin}
Let $\pi\in \cM(G)$ and let $\Xi$ denote the associated variety\footnote{see \S \ref{subsec:not} below for the definition} of the annihilator of $\pi$. If $\bfX$ is $\Xi$-spherical then $\pi$ has  finite multiplicity in $\Sc(X)$, i.e. $$\dim \Hom_G(\Sc(X),\pi)<\infty.$$
\end{introthm}

We actually prove a more general theorem, that allows to consider certain bundles on $X$, and replaces the reductivity assumption on $G$ by certain assumptions on $\pi$.

We deduce the theorem from the following theorem on invariant distributions.

\begin{introthm}[See Theorem \ref{Thm:Dist} below]\label{intThm:Dist}
Let $I\subset \cU(\fg)$ be a two-sided ideal, and let $\cV(I)\subset \fg^*$ denote its associated variety. Suppose that $\cV(I)$ lies in the nilpotent cone of $\fg^*$. Let $\bf X,Y$ be $\cV(I)$-spherical $\bfG$-manifolds.
Let $\Sc^*(X\times Y)^{\Delta G,I}$ denote the space of tempered distributions on $X\times Y$ that are invariant under the diagonal action of $G$, and annihilated by $I$.
Then
$$\dim \Sc^*(X\times Y)^{\Delta G,I}<\infty$$
Moreover, the space $\Sc^*(X\times Y)^{\Delta G,I}$ consists of holonomic distributions.
\end{introthm}

We prove this theorem using the theory of modules over the ring of algebraic differential operators. Namely, we use the theorem that states that the space of solutions of every holonomic D-module in tempered distributions is finite-dimensional.

\DimaH{
\subsection{Homogeneous case }
For $X=G/H$ a version of Frobenius reciprocity (see Lemma \ref{lem:Frob} below)  states $\Hom(\Sc(G/H),\widetilde{\pi})\cong (\pi^*)^H$. 
In this case Theorem \ref{IntThm:fin} can be also deduced from the following one.

\begin{thm}[\cite{Yam}]\label{thm:Yam}
Let $M$ be a finitely generated $\fg$-module and let $\mathfrak{S}(M)$ denote the associated variety\footnote{See \S\ref{sec:Yam} below for the definition.} of $M$. Let $\fh\sub \fg$ be a Lie subalgebra. Assume that $\mathfrak{S}(M)\cap \fh^{\bot}=\{0\}$. 

Then $M$ is finitely generated over $\fh$.
\end{thm}

The relation between the geometric conditions of Theorem \ref{thm:Yam} and Theorem \ref{IntThm:fin}  is given by the next proposition.
Let $\cK_G$ denote the space of maximal compact subgroups of $G$. It has a natural structure of a clopen subset of the real points of a $\bfG$-symmetric variety.
Let $\bf H\sub G$ be an algebraic subgroup, and let $\fh$ denote the Lie subalgebra of $\bfH$. 

\begin{introprop}[\S {\ref{subsec:GeoYam}}]\label{intprop:GeoYam}
Let $\Xi\sub \cN(\fg^*)$ be a closed $\bfG$-invariant subvariety. Suppose that $\bf G/H$ is $\Xi$-spherical. Then there exists a Zariski open dense subset  $U\sub \cK_G$ such that for every $K \in U$ we have $$\Xi \cap \fh^{\bot}\cap\fk^{\bot}=\{0\},$$
where $\fk$ denotes the complexified Lie algebra of $K$.
\end{introprop}

We thus obtain the following strengthening of the homogeneous case of Theorem \ref{IntThm:fin}.

\begin{introcor}[\S {\ref{sec:Yam}}]\label{cor:Yam}
Let $\Xi\sub \cN(\fg^*)$ be a closed $\bfG$-invariant subvariety. Suppose that $\bf G/H$ is $\Xi$-spherical. Then there exists a Zariski open dense $U\sub \cK_G$  such that for every $K \in U$, 
every finite cover $\widetilde{K}$ of $K$, every Harish-Chandra $(\fg,\widetilde{K})$-module $M$ such that the associated variety of the annihilator of $M$ lies in $\Xi$ is finitely generated over $\fh$. \\ In particular, for every finite-dimensional $\fh$-module $\sigma$, and every $i\geq 0$ we have  
$$\dim \Ext^i_{\fh}(M|_{\fh},\sigma)<\infty \text{ and }\dim\Tor^i_{\fh}(M|_{\fh},\sigma)<\infty$$
\end{introcor}

\subsection{Branching problems}
Applying Theorem \ref{IntThm:Ri} and Corollary \ref{cor:Yam} to branching problems we obtain the following statements. 
In this subsection we let $\bf H\sub G$ is a reductive subgroup.
Let $\widetilde{G}$ be a finite cover of an open subgroup of $G$, and let $\widetilde{H}$ be an open subgroup of the preimage of the real points of $\bf H$ in $\widetilde{G}$.
For a nilpotent orbit $\bfO\sub \fg^*$ denote by $\cM_{\overline{\bfO}}(\widetilde{G})$ the subcategory of $\cM(G)$ consisting of representations with associated variety of the annihilator lying in the closure of $\bfO$.

\begin{introcor}[See Proposition \ref{prop:branch} below]\label{IntCor:branch}
Let $\bfO_1\subset \fg^*$ and $\bfO_2\subset \fh^*$ be nilpotent orbits. \DimaC{Consider the following conditions:
\begin{enumerate}[(a)]
%

\item \label{int:mRich} \DimaB{ $\bfO_1=\bf O_P$ for some parabolic subgroup $\bf P\subset G$,  $\bfO_2=\bf O_Q$ for some parabolic subgroup $\bf Q\subset H$, and the set of double cosets $\bf Q\backslash G/P$ is finite. }
\item \label{int:mO12} $\dim \bfO'_1\cap p_{\fh}^{-1}(\bfO'_2)\leq (\dim \bfO'_1+\dim\bfO'_2)/2$ for any $\bfO'_1\subset \overline{\bfO_1}$ and $\bfO'_2\subset \overline{\bfO_2}$.
\DimaH{
\item \label{int:Ext} There exist Zariski open dense subsets $U\sub \cK_G$ and $V\sub \cK_H$ such that for every $K\in U$ and $K'\in V$, and every finite cover $\widetilde{K}$ of $K$ and $\widetilde{K'}$\ of $K'$, every $M\in \cM_{\overline{\bfO_{1}}}(\fg,\widetilde{K})$
and $L\in \cM_{\overline{\bfO_{2}}}(\fh,\widetilde{K'})$, and every $i\geq 0$   we have $\dim \Tor^i_{\fh}(M|_{\fh},L)<\infty$ }
\item \label{int:FinBran}For every $\pi\in \cM_{\overline{\bfO_{1}}}(\widetilde{G})$ and $\tau\in \cM_{\overline{\bfO_2}}(\widetilde{H})$, we have $\dim \Hom_{\widetilde{H}}(\pi|_{\widetilde{H}},\tau)< \infty$.
\end{enumerate}
Then we have \eqref{int:mRich}$\Rightarrow$\eqref{int:mO12}$\Rightarrow$
\eqref{int:Ext}$\Rightarrow$
 \eqref{int:FinBran}.}
\end{introcor}


\begin{introcor}[See \S \ref{subsec:RepBranch} below]\label{IntCor:branch2}
\DimaA{Let $\bf H\sub G$ be a reductive subgroup}.
\begin{enumerate}[(i)]
\item \label{int:FinMultRes} Let $\bf P\sub G$ be a parabolic subgroup, and suppose that $\bf G/P$ is a spherical $\bf H$-variety. Then  for every $\pi\in \cM_{\overline{\bfO_{\bf P}}}({\widetilde{G}})$, the restriction $\pi|_{\widetilde{H}}$ has \DimaH{bounded} multiplicities. \DimaH{This means 
$$\sup_{\tau\in \cM({\widetilde{H}})}\dim \Hom_{\widetilde{H}}(\pi|_{\widetilde{H}},\tau)< \infty.$$
}

\item \label{int:BoundMultResGen} Let $\bfO\sub \fg^*$ be a nilpotent orbit, and $\bf B_H\sub H$ be a Borel subgroup. Suppose that $\bf G/B_H$ is an $\overline{\bfO}$- spherical $\bf G$-variety. Then  for every $\pi\in \cM_{\overline{\bfO}}({\widetilde{G}})$, the restriction $\pi|_H$ has \DimaH{bounded} multiplicities.

\item \label{int:FinMultInd}
Let $\bf Q\subset H$ be a parabolic subgroup. If $\bf Q$ is spherical as a  subgroup of $\bfG$, then for any $\tau \in \cM_{\overline{\bfO_{\bf Q}}}({\widetilde{H}})$, the Schwartz induction $\ind_{\widetilde{H}}^{\widetilde{G}}\tau$ has finite multiplicities.
\DimaH{This means that for every
$\pi\in \cM({\widetilde{H}})$ we have $\dim \Hom(\ind_{\widetilde{H}}^{\widetilde{G}}\tau,\pi)< \infty$.}
\end{enumerate}
\end{introcor}
}

For the case when the commutant $[\fg,\fg]$ is simple, all pairs $\bf (H,P)$ such that $\bfH$ is a symmetric subgroup of $\bfG$, and $\bf G/P$ is a spherical $\bf H$-space are classified in \cite[ \S 5, Table 2]{HNOO}.
More generally, a strategy for classifying all pairs $\bf (H,P)$ of subgroups of $\bf G$ satisfying the conditions of Corollary \ref{IntCor:branch2}\eqref{int:FinMultRes} is given in \cite{AP}. This strategy is also implemented in \emph{loc. cit.} \DimaB{for $\bfG=\SL_n$, and in \cite{AP2} for all the other classical groups}.

The pairs of subgroups $\bf Q\subset H\subset G$ such that $\bf H$ is a symmetric subgroup of $\bf G$, and $\bf Q$ is a parabolic subgroup of $\bf H$ that is also a spherical subgroup of $\bf G$ are classified in \cite[ \S 6, Table 3]{HNOO}.
For some representatives of this class of  pairs, multiplicities in the restriction to $H$ of degenerate principal series representations of $G$ are studied in detail in \cite{MOO,FO}.

%

\DimaA{The main example of non-simple} $[\fg,\fg]$ is the diagonal symmetric pair: $\bf G=H\times H$, with $\bfH $ embedded diagonally. This case gives the following corollary.

\begin{introcor}\label{cor:prod}
Let $\bf H$ be a reductive group, and $\bf P,Q\subset H$ be parabolic subgroups. Suppose that $\bf H/P\times H/Q$ is a spherical $\bf H$-variety, under the diagonal action.

Then for any $\pi\in \cM_{\overline{\bfO_{\bf P}}}(H)$ and any $\tau\in \cM_{\overline{\bfO_{\bf Q}}}(H)$, the tensor product $\pi\otimes \tau$ has \DimaH{bounded} multiplicities as a representation of $H$.
\end{introcor}

Corollaries \ref{IntCor:branch} and \ref{IntCor:branch2} also hold in a wider generality, that allows the groups to be non-reductive, but puts restrictions on the representations (see Definition \ref{def:adm}). This allows to apply Corollary \ref{IntCor:branch2}\eqref{int:FinMultInd} to mixed models. We do so in Corollary \ref{cor:IndSpher} below. Let us give an example for this corollary, that can also be seen as a generalization of the Shalika model.
\begin{introexm}[See \S \ref{subsec:RepBranch} below]\label{Ex:HShal}
Let $\bfG=\GL_{2n}$, $\bf R\subset G$ be the standard parabolic subgroup with Levi part ${\bf L}= \GL_n\times\GL_n$ and unipotent radical $\bf U=\Mat_{n\times n}$,  ${\bf M}=
\Delta \GL_n\subset {\bf L}$, $\bf H:=MU$. Let $\bfO_{\min}\subset \fm^*$ denote the minimal nilpotent orbit (which consists of rank 1 matrices), and let $\pi\in \cM_{\overline{\bfO_{\min}}}(M)$. Extend $\pi$ to a representation of $H$ by letting $U$ act trivially. Let $\psi$ be a unitary character of $H$. Then $\ind_H^G(\pi\otimes\psi)$ has finite multiplicities.
\end{introexm}

A similar example works for the orthogonal groups $\bfG=O_{4n}$, ${\bf L}=\GL_{2n}$, ${\bf M}=\Sp_{2n}$, and the next-to-minimal orbit $\bfO_{\mathrm{ntm}}\sub \fm^*$, which consists of matrices of rank 2  in $\mathfrak{sp}_{2n}^*$.

\subsection{Background and motivation}
Harmonic analysis on spaces with a group action is a central direction of modern representation theory. So far, most of the attention was given to spherical spaces, see e.g. \cite{Del_Sym,vBS,AGRS,AMOT,SZ,GGPW,MgW,KO,Kob,KS,KS_Lect,SV,KKS_Temp,Del,GGP,DKKS,Wan}. Indeed, the spherical (or real/p-adic spherical) spaces $X$ seem to be the most natural spaces to consider if one wants to analyse the entire space of functions on $X$, because of the coherence properties this space possesses, see \cite{KO,KS,SV}, and \cite[Appendix A]{AGS_Z}.

However, if we restrict our attention to a subcategory of representations of the group, namely to representations with associated variety of the annihilator lying in a fixed subset of the nilpotent cone,  some coherence properties hold in a wider generality, as exhibited by Theorem \ref{IntThm:fin}. This serves as our main motivation for the notion of a $\Xi$-spherical space.
\subsection{Examples}

The classification of all pairs of parabolic subgroups satisfying the condition of Corollary \ref{cor:prod} is given in \cite{Stem}. In particular, this shows that the product of two small representations\footnote{{\it i.e.} representations such that the square of every matrix in the associated variety of the annihilator is zero.} of a classical group has finite multiplicities.
Also, for $\GL_n$ the product of any representation $\pi\in \cM(G)$ with any minimal representation $\tau$ has finite multiplicities. This allows to define some analogue of translation functors, by sending $\pi$ to the projection of $\pi\otimes \tau$ on the subcategory corresponding to a fixed central character.

For $\bfH\in \{\GL_n,\Sp_{2n},SO_{2n+1}\}$, all triples of parabolic subgroups such that $\bf H$ has finitely many orbits under the diagonal action on $\bf H/P_1\times H/P_2\times H/P_3$ are classified in \cite{MWZ_GL,MWZ_Sp,Matsuki}. They also show that these groups never have finitely many orbits on the quadriple product $\bf \prod_{i=1}^4 H/P_i$ unless $\bf P_i=H$ for some $i$.


Given  an orbit $\bfO \subset \mu(\bf X)\subset \fg^*$  it is natural to ask whether a strict inequality $\dim \mu^{-1}(\bfO)<\dim X + \dim \bfO/2$ is possible. In Proposition \ref{prop:RichDim} below we show that it is not possible if $\bfO$ is a Richardson orbit. However, in Appendix \ref{app:Ido} below, Ido Karshon shows that for non-Richradson $\bfO$ it is possible. In this example $\bf X=G/H$, and in fact it is computed that $\dim \bfO\cap \fh^{\bot}<\dim \bfO/2$. This is interesting since the scheme-theoretic intersection of $\bfO$ with $\fh^{\bot}$ is co-isotropic with respect to the Kirillov-Kostant-Souriau symplectic form on $\bfO$ \DimaC{(see Definition \ref{def:sympO} below)}, but obviously in this case the intersection is not reduced. \DimaB{Another example, in which $\bfH$ is a parabolic subgroup, was provided by Dmitry Panyushev - see Example \ref{ex:small} below. Under additional conditions that $\bf H$ is either symmetric, or spherical and solvable this is not possible by \cite[Proposition 3.1.1]{Gin} and \cite[Theorem 1.5.7]{CG} which show that in these cases the set-theoretic intersection $\bfO\cap \fh^{\bot}$ is Lagrangian.}

Let us now give an example of an $\overline{\bfO}$-spherical subgroup for a non-Richardson orbit $\bfO$, computed in \cite[\S 9 and Appendix A]{GS}, motivated by the local $\theta$-correspondence in type II.

\begin{example}
Let $\bfG:=\GL_n\times \GL_k\times \Sp_{2nk}$ (when we pass to real groups we will have to consider a double cover). Let $\iota$ be the composition of the two natural embeddings
$$\GL_n\times \GL_k \into \GL_{nk}\into \Sp_{2nk},$$
and let $\bfH$ be the graph of $\iota$.
Let $\bfO:=\bfO_{\mathrm{reg}}\times \bfO_{\mathrm{reg}}\times \bfO_{\min}$. Then $\bf G/H$ is  $\overline{\bfO}$-spherical.
\end{example}
\DimaI{Thus Corollary \ref{IntCor:branch2} implies that for any irreducible $\pi\in \cM_{\overline{\bfO_{\min}}}(\Mp_{2nk}(\R))$, the restriction $\pi|_{\GL_n(\R)\times \GL_k(\R)}$ has bounded multiplicities. Here, $\Mp_{2nk}(\R)$ denotes the metaplectic group. 
Actually, the theory of local $\theta$-correspondence implies that these multiplicities are bounded by one. }
\subsection{Related results}

It is proven in \cite{KO} that for spherical subgroups $\bf H\subset \bf G$, the representation $\Sc(G/H)$ has bounded 
multiplicities. The literal analogue of this result to $\overline{\bfO}$-spherical subgroups  cannot hold already for $\overline{\bfO}=\{0\}$:
every subgroup $\bf H\subset G$ is $\{0\}$-spherical, including the trivial $\bfH$,
$\cM_{\{0\}}(G)$ consists of finite-dimensional representations, and multiplicities equal the dimension.
A possible general conjecture is that it is bounded by $c_{\bfO\cap \fk^{\bot}}$- some total multiplicity (taken with some coefficients of geometric nature) of $K$-orbits lying \DimaA{in $\bfO\cap \fk^{\bot}$ in the associated cycle of $\pi$ (which is a finer invariant than the annihilator variety), where $K$ is a maximal compact subgroup of $G$}.

\DimaA{On the other hand, \cite[\S 7]{Kob2}, \cite{Kob3} give a sufficient condition for bounded multiplicities for branching problems for symmetric pairs, in terms of distinction with respect to another symmetric subgroup. By Theorem \ref{thm:GS} below and \cite{Kno} this condition is more restrictive than the condition in Corollary \ref{IntCor:branch2}\eqref{int:FinMultRes}}.

\DimaH{The recent work \cite{Kit} gives sufficient geometric conditions for bounded multiplicities of restriction and induction, also in terms of associated varieties of the annihilator. While these conditions are also related to the notions of spherical varieties, they are in general a-priory different from those in Corollary \ref{IntCor:branch} and Corollary \ref{IntCor:branch2}\eqref{int:BoundMultResGen}. In case Richardson orbits, {\it i.e.} Corollary \ref{IntCor:branch2}\eqref{int:FinMultRes},\eqref{int:FinMultInd}, our conditions are equivalent to those in \cite{Kit}. However, in case of Corollary \ref{IntCor:branch2}\eqref{int:FinMultInd} the result in \cite{Kit} is stronger since it gives bounded multiplicities for $\ind_{H}^G\tau$ rather than finite. On the other hand, in Corollary \ref{IntCor:branch} we also show finiteness of the dimension of higher $\Tor$. }

By Theorems \ref{IntThm:Ri} and \ref{IntThm:fin}, a sufficient condition for every $\pi \in \cM_{\overline{\bfO_P}}(G)$ to have finite multiplicity in $\Sc(G/H)$ is that $\bf P$ has finitely many orbits on $\bf G/H$. 

The following result provides a necessary condition in similar terms.

\begin{thm}[{\cite[Theorem 2.4]{TauchiSurvey}}]\label{thm:RiNess}
Let $P\subset G$ be a parabolic subgroup. If all degenerate principal series representations of the form $\Ind_P^G\fin$, with $\dim \fin<\infty$  have finite multiplicities \DimaG{ in $\Sc(G/H)$}, then $H$ has finitely many orientable orbits on $G/P$.
\end{thm}

\begin{introcor}
Let $\bf P\sub G$ be a parabolic subgroup defined over $\R$, and let $P$ be the corresponding parabolic subgroup of $G$. Suppose that for all but finitely many orbits of $\bf H$ on $\bf G/\bfP$, the set of real points is non-empty and orientable. Then the following are equivalent.
\begin{enumerate}[(i)]
\item $\bf G/H$ is $\overline{\bfO_{\bfP}}$-spherical, where $\bfO_{\bfP}$ denotes the Richardson orbit of $\bf P$.
\item Every $\pi\in \cM_{\overline{\bfO_{\bfP}}}(G)$ has finite multiplicity \DimaG{ in $\Sc(G/H)$}.
\item $H$ has finitely many orbits on $G/P$.

\item $\bf H$ has finitely many orbits on $\bf G/P$.

\end{enumerate}

\end{introcor}
The assumption of the corollary holds in particular if $H$ and $G$ are complex reductive groups.

\DimaA{Another necessary condition for every $\pi \in \cM_{\overline{\bfO_P}}(G)$ to have finite multiplicity in $\Sc(G/H)$ is that $H$ has an open orbit on $G/P$, see \cite[Corollary 6.8]{Kob1}.}

When the conditions of the corollary are not satisfied, the finiteness of $\bf H\backslash G/P$ is not necessary. Indeed, when $\bf P$ is a minimal parabolic subgroup defined over $\R$, the finiteness of $ H\backslash G/P$ implies that  \DimaG{$\Sc(G/H)$ has finite multiplicities}, by \cite{KO}. However, for general parabolic subgroups, the finiteness of $ H\backslash G/P$ is not sufficient, and a series of counterexamples is provided in \cite{TauchiExample}.
These are examples of non-spherical $\bf H\subset G$, and a parabolic $\bf P\subset G$ for which $ H\backslash G/P$ is finite, but the multiplicities \DimaG{in $\Sc(G/H)$} of degenerate principal series representations $\Ind_P^G\chi$ are infinite.
For a very explicit description of the basic example in these series see \cite[Outline of the proof of Theorem 2.2]{TauchiSurvey}.


The recent work \cite{GS} allows also to give a microlocal necessary condition for the non-vanishing of the multiplicity space. For this purpose we remark that by \cite{BB3,Jos} (cf. \cite[Corollary 4.7]{Vog}), for any irreducible $\pi\in \cM(G),$ the \DimaG{associated variety of the annihilator $\cV(\Ann\pi)$ is the closure of a single orbit, that we will denote $\bfO(\pi)$.}

\begin{theorem}[{\cite[\S 9]{GS}}]\label{thm:GS}
Let $\pi\in \cM(G)$ be irreducible.
Suppose that   $\pi$ is a quotient of $\Sc(X)$, for some $\overline{\bfO(\pi)}$-spherical $\bfG$-manifold $\bfX$. Then $\bfO(\pi)\subset \Im \mu(\bfX)$. Moreover,
if $\bf X=G/H$ then $$2\dim\bfO(\pi)\cap \fh^\bot=\dim\bfO(\pi).$$
\end{theorem}

Theorem \ref{intThm:Dist} implies that the relative characters of $\pi\in \cM_{\Xi}(G)$ corresponding to $\Xi$-spherical subgroups are holonomic. For spherical subgroups, it is further shown in \cite{Li} that the relative characters are regular holonomic. We conjecture that this holds for $\Xi$-spherical subgroups as well.

\subsection{Open questions}

The first open question is to give a necessary and sufficient condition for finite multiplicities \DimaG{in $\Sc(X)$} for all representations in $\cM_{\Xi}(G)$. As discussed above, in some settings our sufficient condition is also necessary, but in others it is not.

Further, we can consider an ``additive character'' $\chi$ of $\fh$, i.e. a differential of a group map $\bfH\to G_a$. Then, we think that under some conditions on $\fh$ and $\chi$, in all the statements above we can replace the multiplicity spaces by $\Hom_{\fh}(\pi,\chi)$, and the set $\fh^{\bot}$ by $p_{\fh}^{-1}(\chi)$, where $p_{\fh}:\fg^*\to \fh^*$ is the standard projection.
This would imply the finiteness of certain generalized Whittaker models.
Such a twisted version of Theorem \ref{thm:GS} is proven in \cite[\S 9]{GS}. 


We would also like to find an example in which $\bfO\cap \fh^{\bot}$ has  dimension \DimaB{at most} $\dim \bfO/2$, but for some $\bfO'\subset \overline \bfO$ we have $\dim \bfO' \cap \fh^{\bot}> \dim \bfO'/2$, or prove that such examples do not exist. \DimaB{We would also like to know whether the $\overline{\bfO}$-sphericity is equivalent to a certain subvariety of $T^*\bfX\times \bfO'$ being isotropic, for every $\bfO'\subset \overline \bfO$, as we prove for  Richardson $\bfO$ in Theorem \ref{Thm:Ri} below.}

Finally, we are very much interested in the non-archimedean analogues of our results.

\subsection{Structure of the paper}
In \S \ref{sec:Geo} we prove Theorem  \ref{IntThm:Ri}, as well as some other geometric results needed for the corollaries on branching problems.

In \S \ref{sec:Dmod} we prove Theorem \ref{intThm:Dist} by showing that the system of differential equations satisfied by the distributions in question is holonomic.

In \S \ref{sec:main} we deduce Theorem \ref{IntThm:fin} from Theorem \ref{intThm:Dist}, and then deduce Corollaries \ref{IntCor:branch} and \ref{IntCor:branch} from Theorem \ref{IntThm:fin} and the geometric results in \S \ref{sec:Geo}.
In Appendix \ref{app:Pfms} we prove a geometric proposition \ref{prop:ms} on non-reductive groups.
Finally, in Appendix \ref{app:Ido} we give an example in which the inequality in Definition \ref{def:Ospher} is strict.

\subsection{Acknowledgements}
We thank Joseph Bernstein, Shachar Carmeli, \DimaH{Dor Mezer,} Eitan Sayag, \DimaC{Dmitry A. Timashev} for fruitful discussions.  \DimaC{We also thank Dmitri I. Panyushev for finding a mistake in a previous version of this paper and providing Example \ref{ex:small} below}.

\section{Geometry}\label{sec:Geo}
\subsection{Preliminaries and notation}\label{subsec:GeoPrel}
From now and till the end of the section we let $\bfG$ be a connected complex linear algebraic group. We will be mainly interested in the case of reductive $\bfG$. Let $\fg$ denote the Lie algebra of $\bfG$, and $\fg^*$ denote the dual space.

In general, we will denote algebraic groups by boldface letters, and their Lie algebras by the corresponding Gothic letters.

\DimaC{We fix an algebraic $\bfG$-manifold $\bfX$}. We start with the definition of the moment map.
For any point $x\in \bfX$, let $a_x:\bfG\to \bfX$ denote the action map, and $da_x:\fg\to T_x\bfX$ denote its differential.
The moment map $\mu:=\mu_{\bf X}:T^*\bfX\to \fg^*$ is defined by
$$\mu_{\bf X}(x,\xi)(\alp):= \xi(da_x(\alp))$$

Denote by $\bf U_G\sub G$ the unipotent radical of $\bf G$, and by $\fu_{\fg}^{\bot}\sub \fg^*$ the space of functionals on $\fg$ that vanish on $\fu_{\fg}$.
We will call an element $\varphi\in \fg^*$ \emph{nilpotent} if $\varphi\in \fu_{\fg}^{\bot}$, and the closure of the coadjoint orbit $\bfG\cdot \varphi$ includes 0.
Let $\cN(\fg^*)\subset \fg^*$ denote the nilpotent cone. By Kostant's theorem, $\cN(\fg^*)$ consists of finitely many coadjoint $\bfG$-orbits. Indeed, these orbits are in bijection with the nilpotent  coadjoint orbits of the reductive quotient $\bf R_G:=G/U$ under the identification $\fr_{\fg}^*\cong \fu_{\fg}^{\bot}\sub \fg^*$.
The requirement $\varphi\in \fu_{\fg}^{\bot}$ is motivated by the following proposition.
\begin{prop}[See Appendix \ref{app:Pfms} below]\label{prop:ms}
Let $\Xi\subset \fg^*$ be a closed conical subset that is a union of finitely many orbits. Then $\Xi\subset \fu_{\fg}^{\bot}$.
\end{prop}

We will say that a subgroup $\bf P\sub G$ is \emph{parabolic} if it is the preimage of a parabolic subgroup of $\bf R_{G}$ under the projection $\bf G \onto R_{G}$.

\begin{defn}
Let $\bf P\subset G$ be a parabolic subgroup, and let $\fp^{\bot}\subset \fg^*$ denote the space of functionals vanishing on $\fp$. It is easy to see that $\fp^{\bot}\sub \cN(\fg^*)$, and thus there exists a unique nilpotent orbit that intersects $\fp^{\bot}$ by an open dense subset. It is called the Richardson orbit, and we will denote it by $\bfO_{\bfP}$, and its closure by $\overline{\bfO_{\bfP}}$.
\end{defn}
It is easy to see that \DimaB{$\overline{\bfO_{\bfP}} = \bfG \cdot \fp^{\bot}=\Im(\nu_{\bf P}),$ where $\nu_{\bf P}:T^*({\bf G/P})\to \fg^*$ denotes the moment map of $\bf G/P$.} For more information on Richardson orbits we refer the reader to \cite[\S 7.1]{CM}.

The definition of $\bfO$-sphericity and $\Xi$-sphericity for non-reductive groups is identical to the one given in Definition \ref{def:Ospher}.
Theorem \ref{IntThm:Ri} is also valid for non-reductive groups, and the proof below works in this generality.


\DimaB{
\subsubsection{Preliminaries on symplectic manifolds}
A \emph{symplectic manifold} is a manifold with a \emph{symplectic form}, {\it i.e.} a closed non-degenerate 2-form. We will call a subvariety $\bf Z$ of a symplectic \DimaC{algebraic} manifold $\bf M$ \emph{isotropic} if  the restriction of the symplectic form to the tangent space to $\bf Z$ at every  smooth point of $\bf Z$ is zero. Similarly, $\bf Z$ is called \emph{coisotropic} if at every smooth point $z$ of $\bf Z$, the tangent space $T_{z}{\bf Z}$ includes its symplectic complement $(T_{z}{\bf Z})^{\bot}\subset T_z{\bf M}$, and $\bf Z$ is called Lagrangian if it is both isotropic and coisotropic. \DimaC{We refer the reader to \cite[\S 1.3]{CG} for more details on these notions. In particular, we will need the following proposition.
}

\DimaD{

\begin{prop}[{\cite[Proposition 1.3.30 and \S 1.5.16]{CG}}]\label{prop:subIs}
If \DimaC{a subvariety} $\bf Y\sub M$ is isotropic then so is any subvariety $\bf Z\sub Y$.
\end{prop}

The statement is nontrivial because $\bf Z$ may be contained in the set of singular points of $\bf Y$.
}

Let us now give two standard examples of symplectic manifolds that will play an important role in the next section.

\begin{defn}
Let $\bf X$ be a manifold. Define a 1-form $\theta$ on $T^*\bfX$ at any point $(x,\lam)\in T^*\bfX$ by $\theta(\xi):=\lam(d_{(x,\lam)}p_X(\xi))$, where $p_X:T^*\bfX\to \bfX$ is the natural projection, and $d_{(x,\lam)}p_X$ is its differential at the point $(x,\lam)$. The natural symplectic form $\omega$ on  $T^*\bfX$ is defined to be the differential of the 1-form $\theta$.
\end{defn}

\begin{lem}\label{lem:0}
 Let $\bfS:=\mu_X^{-1}(\{0\})$.
Then the following are  equivalent:
\begin{enumerate}[(i)]
\item \label{it:Hfin} $\bfG$ has finitely many orbits on $\bfX$.
\DimaB{\item \label{it:Lag} $\bfS$ is a Lagrangian subvariety of $T^*\bfX$.}
\item \label{it:eqDim} $\bfS$ is equidimensional of dimension $\dim \bfX$.
\item \label{it:smallDim} $\dim \bfS\leq \dim \bfX$
\end{enumerate}
\end{lem}
\begin{proof}
\eqref{it:Hfin}$\Rightarrow$\eqref{it:Lag}:
The set $\bfS$ is the union of conormal bundles to orbits. The conormal bundle to each orbit is irreducible \DimaB{and Lagrangian.\\
\eqref{it:Lag}$\Rightarrow$\eqref{it:eqDim} and \eqref{it:eqDim}$\Rightarrow$\eqref{it:smallDim} are obvious.\\}
\eqref{it:smallDim}$\Rightarrow$\eqref{it:Hfin}:
Denote by $p_\bfX:T^*\bfX\to \bfX$ the natural projection.
For every $\bfG$-orbit $\bfR\subset \bfX$ we have $\bfS\cap p_\bfX^{-1}(\bfR)=CN_\bfR^\bfX$.
By Rosenlicht's theorem, there is an open non-empty subset $\bf U\subset \bfX$ that has a geometric quotient by $\bfG$. Applying this theorem again to the complement $\bf \bfX\setminus U$, and further by induction,  we obtain a stratification of $\bfX$ by such sets $T_i$. If for some $i$, $T_i$ is not a finite union of orbits, then   $\dim \bfS\cap p_\bfX^{-1}(T_{i})>\dim \bfX$, contradicting the assumption.
\end{proof}

\begin{exm}[{cf. \cite[p. 211]{Gin}}]\label{ex:T*X}
Let us compute the form on $T^*({\bf G/H})$, for an algebraic subgroup $\bf H\sub G$ at a point $\lam=(1,a)$, where $a$ is a point in the cotangent space to $\bf G/H$ at the base point $1\in {\bf G/H}$. We can write any tangent vector at $\lam$ to $ T^*({\bf G/H})$ in the form $x\cdot \lam - \alp$, where $x\in \fg$, $x\cdot \lam$ denotes the action of $x$ on $\lam$, and   is a ``vertical'' tangent vector:
$$\alp\in \fh^{\bot} \cong T_{a}\fh^{\bot} \cong T_{a}T^*_{1}({\bf G/H})\sub T_{(1,a)}T^*({\bf G/H}).$$

 Then the standard symplectic form on $T^*_{}({\bf G/H})$ is given at $\lam=(1,a)$ by
$$\omega(x\cdot \lam, y\cdot \lam)=\langle a, [x,y]\rangle, \quad \omega(x\cdot \lam, \alp)=\langle \alp, x \rangle,\quad  \omega(\alp,\beta)=0 \quad \forall  x,y \in \fg, \,\alp,\beta\in \fh^{\bot}$$
\end{exm}

\begin{defn}\label{def:sympO}
Let $\bfO\sub \fu_{\fg}^{\bot}\subset \fg^*$ be a coadjoint orbit, and let $a \in \bfO$. The Kirillov-Kostant-Souriau symplectic form on $\bfO$ is given at $a$ by $$\omega_{\bfO}(\ad^*(x)( a),\ad^*(y)(a))=\langle a, [x,y]\rangle.$$
\end{defn}

\DimaD{
\begin{prop}[{\cite[Theorem 3.3.7]{CG}}]\label{prop:Piso}
For any orbit $\bf O\sub \overline{O_P}$, the (set-theoretic) intersection  $\bfO\cap \fp^{\bot}$ is an isotropic subvariety of $\bfO$. Moreover, if $\bf P\sub G$ is a Borel subgroup then  $\bfO\cap \fp^{\bot}$ is Lagrangian in $\bfO$.
\end{prop}
}

}

\subsection{$\Xi$-sphericity criteria}\label{subsec:crit}




%
%
%

\DimaC{
Let us give a quantitative version of the notion of $\bfO$-sphericity.
\begin{defn}
For any nilpotent coadjoint orbit $\bfO\subset \fu_{\fg}^{\bot}\sub \fg^*$, we define the $\bfO$-complexity of $\bfX$ to be $$c_{\bfO}(\bfX):=\dim\mu^{-1}(\bfO)-\dim\bfX - \dim\bfO/2.$$
For a $\bfG$-invariant subset $\Xi \sub \cN(\fg^*)$ we define
$$c_{\Xi}(\bfX):=\max_{\text{orbit }\bfO\sub \Xi}c_{\bfO}(\bfX)$$
\end{defn}

\DimaE{If the orbit $\bfO$ does not intersect $\mu(\bfX)$, we say that $c_{\bfO}(\bfX)=-\infty$.
Even if $\bfO$ intersects $\mu(X)$,  $c_{\bfO}(\bfX)$ can  be negative, see Example \ref{ex:small} below. However, it is easy to see that $c_{\{0\}}(\bfX)\geq 0$, and thus for every closed $\bfG$-invariant subset $\Xi \sub \cN(\fg^*)$ we have $c_{\Xi}(\bfX)\geq 0$. Also, for Richardson orbits $\bfO$ we have }$c_{\bfO}(\bfX)\geq 0$, as we show in Proposition \ref{prop:RichDim} below.
Let us now give a more explicit formula for the $\bfO$-complexity of homogeneous varieties.
Let $\bfH\sub \bfG$ be an algebraic subgroup, and let $\fh$ be its Lie algebra. Let $\fh^{\bot}\subset \fg^*$ denote the space of functionals vanishing on $\fh$.
For any $g\bfH\in \bfG/\bfH$, identify the cotangent space to $\bfG/\bfH$ at $g\bfH$ with $g\cdot \fh^{\bot}\subset \fg^*$.
Under this identification the moment map sends $(g\bfH,a)\in T^*(\bfG/\bfH)$ to $a \in \fg^*$. Thus, the image of $\mu$ is $\bfG\cdot \fh^{\bot}$.

\DimaE{
\begin{lemma}\label{lem:inO}
For any nilpotent coadjoint orbit $\bfO\sub \fg^*$, there exists a map
$\alp:\mathbf{O}\cap \fh^{\bot}\to \mu^{-1}(\bfO)$, such that the image of $\alp$ intersects every irreducible component of $\mu^{-1}(\bfO)$, and for every $\varphi\in \mathbf{O}\cap \fh^{\bot},$ we have
$$\dim_{\varphi}(\mathbf{O}\cap \fh^{\bot})+\dim {\bf G/H}=\dim_{\alp(\varphi)}\mu^{-1}(\bfO)$$
\end{lemma}
\begin{proof}
$\mu^{-1}(\mathbf{O})=\{(g\bfH,\xi)\, \vert \xi \in \mathbf{O} \cap g\cdot \fh^{\bot}\}$.
Let $p:\mu^{-1}(\mathbf{O})\to {\bf G/H}$ denote the projection.
We get an identification $p^{-1}([1])\cong \mathbf{O}\cap \fh^{\bot}$ which gives the desired map $\alp$.\end{proof}

\begin{cor}\label{cor:invO}
For any nilpotent coadjoint orbit $\bfO\sub \fg^*$, we have
$$c_{\bfO}({\bf G/H})= \dim \mathbf{O}\cap \fh^{\bot}- \dim\bfO/2.$$
In particular, $\bfG/\bfH$ is $\bfO$-spherical if and only if $\dim \bfO \cap \fh^{\bot}\leq \dim \bfO/2$.
\end{cor}
}


We will say that $\bfH$ is $\bfO$-spherical if $\bfG/\bfH$\ is $\bfO$-spherical, and similarly for any $\bfG$-invariant subset $\Xi\subset \cN(\fg^*)$.

Now we would like to prove Theorem \ref{IntThm:Ri}.}
\DimaB{Let $\bfP\subset \bfG$ be a parabolic subgroup, let $\bfO_\bfP$  denote the corresponding Richardson orbit, and $\overline{\bfO_\bfP}$ denote its closure.
Let $\mu:  T^*\bfX\to \fg^*$ and $\nu:T^*({\bf G/P})\to \fg^*$ be the moment maps.

\begin{prop}\label{prop:symp}
Let $\bf O\sub \overline{O_P}$ be an orbit. Let $\widetilde{\bfO}:=\nu^{-1}(\bfO)$, and let $\nu_{\bfO}:\widetilde{\bfO}\to \bfO$ denote the restriction of $\nu$.
\begin{enumerate}[(i)]
\item The smooth locus of the map $\nu_{\bfO}$ is open and dense in $\widetilde{\bfO}$. \label{it:SpringSmooth}
\item \label{it:SpringSymp} For every smooth point $\lam$ of the map $\nu_{\bfO}$, the pullback of the symplectic form on $\bfO$ under the differential  $d_\lam\nu_O$ equals the restriction of the symplectic form of $T_\lam T^*({\bf G/P})$ to $T_\lam\widetilde{\bfO}$.
\end{enumerate}
\end{prop}

For the proof of \eqref{it:SpringSmooth} we will need the following standard lemma.
\begin{lem}\label{lem:dom}
Let $\pi:{\bf Y\to Z}$ be a dominant map of irreducible varieties. Then the smooth locus of $\pi$ is open and dense in $\bf Y$.
\end{lem}

\DimaG{
For the proof of \eqref{it:SpringSymp} we will need the following lemma.
\begin{lem}\label{lem:pullback}
Let $\lam=(1,a)\in\widetilde{\bfO}\sub T^*({\bf G/P}),$ where $1\in {\bf G/P}$ is the base point, and $a\in T_1^*({\bf G/P})\cong \fp^{\bot}$. Then 
\begin{enumerate}[(i)]
\item $\lam$ is a smooth point of $\widetilde{\bfO}$ if and only if $a$ is a smooth point of $\bfO\cap \fp^{\bot}$.
\item $T_{\lam}(\nu^{-1}(\bfO))=\fg\cdot {\lam} +T_a(\bfO\cap \fp^{\bot})$
\end{enumerate}
\end{lem}
\begin{proof}
Follows from the Cartesian square
$$
\xymatrix{
\bfG\times \fp^{\bot}  \ar[r] &T^*({\bf G/P})\\
\bfG\times ({\bfO \cap \fp^{\bot}})\ar@{^{(}->}[u] \ar[r] &\widetilde{\bfO}\ar@{^{(}->}[u]}
$$

\end{proof}

\begin{proof}[Proof of Proposition \ref{prop:symp}]
\eqref{it:SpringSmooth}: Since $\bfG$ is connected, any irreducible component of $\widetilde{\bfO}$ is $\bfG$-invariant. Hence $\nu_{\bfO}$ maps it onto $\bfO$. The assertion follows now from Lemma \ref{lem:dom}.  \\
\eqref{it:SpringSymp}: 
We can assume that $\lam=(1,a)\in\widetilde{\bfO}\sub T^*({\bf G/P})$ as above. By Lemma \ref{lem:pullback} we can write any tangent vector $v\in T_\lam  (\widetilde{\bfO})$ as 
$v=x\cdot \lam - \alp$, where $x\in \fg$, the expression $x\cdot \lam$ denotes the action of $x$ on $\lam$, and  $\alp\in T_{a}(\bfO \cap \fp^{\bot})$ is a ``vertical'' tangent vector. We have $a=\nu(\lam)\in \bfO$.


Let  $w\in T_\lam  (\nu^{-1}(\bfO))$ be another tangent vector. Decompose it similarly as 
$w=y\cdot \lam - \beta$. By Proposition \ref{prop:Piso}
we have $\omega_{\bfO}(\alp,\beta)=0$. By Example \ref{ex:T*X} and Definition \ref{def:sympO} we have
\begin{multline*}
\omega(x\cdot \lam - \alp, y\cdot \lam - \beta)=\langle a ,[x,y]\rangle - \langle\ad^*(y')(a) , x\rangle+\langle \ad^*(x')(a), y\rangle=
\langle a, [x-x',y-y']\rangle-\langle a, [x',y']\rangle=\\ \omega_{\bfO}(\ad^{*}(x-x')(a),\ad^*(y-y')(a))-\omega_{\bfO}(\alp,\beta)=\\
\omega_{\bfO}(d_\lam\nu_{\bfO}((x-x')\cdot \lam),d_\lam\nu_{\bfO}((y-y')\cdot \lam)-0=\omega_{\bfO}(x\cdot \lam - \alp, y\cdot \lam - \beta)
\end{multline*}
\end{proof}
}
%

Let us now prove the following theorem, that is stronger than Theorem \ref{IntThm:Ri}.

\begin{thm}\label{Thm:Ri}

The following are equivalent.
\begin{enumerate}[(i)]
\item $\bfP$ has finitely many orbits on $\bfX$. \label{it:RiFin}
\item For every orbit $\bf O \sub \overline{O_P}$, the subvariety $\Gamma_{\bf O}:=\{(x,-\mu(x))\in T^*(X)\times \fg^*\, \vert \, \mu(x)\in \bfO\} $ of $T^*(\bfX)\times \bfO$ is isotropic. \label{it:RiIsot}
\item $\bfX$ is $\overline{\bfO_\bfP}$-spherical. \label{it:RiOSpher}
\end{enumerate}
\end{thm}


\begin{proof}

\eqref{it:RiFin}$\Rightarrow$\eqref{it:RiIsot}:
Let ${\bf T_O}\subset T^*\bfX \times T^*({\bf G/P})$ denote the preimage of $\Gamma_{\bf O}$ under $\Id \times \nu$. By Proposition \ref{prop:symp}, it is enough to prove that $\bf T_{\bfO}$ is isotropic.
Note that $$T_{\bfO}\sub T:=\{(x,y)\in T^*\bfX \times T^*({\bf G/P})\, \vert\, \mu(x)=-\nu(y)\}=(\mu\times\nu)^{-1}(\Delta\fg)^{\bot},$$
where $\Delta\fg\sub \fg\times \fg$ is the diagonal, and $(\Delta\fg)^{\bot}\sub (\fg\times \fg)^*$ is its orthogonal complement. Thus $\bf T$ is also the  preimage of zero under the moment map corresponding to the diagonal action of $\bf G$ on $\bf X\times (G/P)$. Since $\bf P$ has finitely many orbits on $\bf X$, $\bf G$ has finitely many orbits on  $\bf X\times (G/P)$ and thus Lemma \ref{lem:0} implies that $\bf T$ is Lagrangian.

\eqref{it:RiIsot}$\Rightarrow$\eqref{it:RiOSpher}: $\dim \mu^{-1}(\bfO)=\dim \Gamma_{\bfO}\leq (\dim T^*\bfX+\dim \bfO)/2=\dim \bfX+\dim\bfO/2$, since $\Gamma_{\bfO}$ is isotropic by \eqref{it:RiIsot}.

\eqref{it:RiOSpher}$\Rightarrow$\eqref{it:RiFin}:
Let $\bfS:=\mu^{-1}(\fp^{\bot})$. Note that  $\bfS$ also equals the preimage of zero under the moment map for the action of $\bfP$ on $\bfX$. Thus, Lemma \ref{lem:0} implies that $\bfP$ has finitely many orbits on $\bfX$ if and only if $\dim \bfS\leq\dim \bfX$.
Further, $\dim \bfS=\max_{\bfO\sub \overline{\bfO_\bfP}} \dim \mu^{-1}(\fp^{\bot}\cap \bfO)$. Since the map $\mu$ is $\bfG$-equivariant, for any orbit $\bfO$, the fibers of all its points are isomorphic. Thus
$$\dim \mu^{-1}(\fp^{\bot}\cap \bfO)=  \dim \mu^{-1}(\bfO)+\dim\fp^{\bot}\cap \bfO-\dim \bfO$$
By Proposition \ref{prop:Piso} , for every $\bfO\sub \overline{\bfO_\bfP}$ we have
$\dim\fp^{\bot}\cap \bfO\leq \dim \bfO/2$. Thus $$\dim\fp^{\bot}\cap \bfO-\dim \bfO\leq -\dim \bfO/2,$$ and thus
$$\dim \mu^{-1}(\fp^{\bot}\cap \bfO)\leq  \dim \mu^{-1}(\bfO)-\dim \bfO/2\leq \dim \bfX,$$
since $\bf X$ is $\bfO$-spherical.  Thus $\dim \bfS\leq \dim \bfX$ and thus $\bfP$ has finitely many orbits on $\bfX$. \end{proof}
The implication \eqref{it:RiFin}$\Rightarrow$\eqref{it:RiIsot} was proven in a special case in \cite[Theorem 3.8]{Li}, and the proof we give is merely a  modification of the proof in \cite{Li}.
}

\DimaE{
Given this theorem, it is natural to ask if $c_{ \overline{\bf O_P}}(\bfX)$ equals the modality of the action of $\bf P$ on $\bfX$ (in the sense of \cite[p. 2]{Vin}). The argument in the proof of \eqref{it:RiOSpher}$\Rightarrow$\eqref{it:RiFin} allows to answer this in the case when $\bfP$ is a Borel subgroup, obtaining the following result regarding the complexity of $\bfX$ in the sense of \cite{Vin}.

\begin{prop} If $\bfG$ is reductive
then the complexity of $\bfX$  equals $c_{\cN(\fg^*)}(X)$.
\end{prop}
\begin{proof}
Let $\bf B\sub G$ be a Borel subgroup, $\fb$ be the Lie algebra of $\bf B$, and  $\bfS:=\mu^{-1}(\fb^{\bot})$. Then by \cite[Theorem 2]{Vin} (see also \cite{Bri}), the complexity of $\bf X$ equals $\dim \bfS-\dim \bf X$.
Since $\fb^{\bot}$ lies in the nilpotent cone, we have $$\dim \bfS=\max_{\bfO\sub \cN(\fg^*)} \dim \mu^{-1}(\fb^{\bot}\cap \bfO).$$ Since the map $\mu$ is $\bfG$-equivariant, the fibers of all points in any orbit $\bfO$ are isomorphic. Thus
$$\dim \mu^{-1}(\fb^{\bot}\cap \bfO)=\dim\fb^{\bot}\cap \bfO+  \dim \mu^{-1}(\bfO)-\dim \bfO.$$
By Proposition \ref{prop:Piso} , for every $\bfO\sub \cN(\fg^*)$ we have
$\dim\fb^{\bot}\cap \bfO= \dim \bfO/2$. Thus
$$\dim \mu^{-1}(\fb^{\bot}\cap \bfO)=  \dim \mu^{-1}(\bfO)-\dim \bfO/2=c_{\bfO}(\bfX)+ \dim \bfX.$$
Thus
$$\dim \bfS-\dim \bf X=\max_{\bfO\sub \cN(\fg^*)} \dim \mu^{-1}(\fb^{\bot}\cap \bfO)-\dim \bfX=\max_{\bfO\sub \cN(\fg^*)}c_{\bfO}(\bfX)= c_{\cN(\fg^*)}(X).$$
\end{proof}

\begin{cor}\label{cor:AbsSpher}
If $\bfG$ is reductive then the following are equivalent:
\begin{enumerate}[(a)]
\item $\bfX$ is $\bfO$-spherical for every $\bfO$.
\item
 $\bfX$ is spherical.
 \end{enumerate}
\end{cor}
}
A similar argument gives the following statement for Richardson orbits.

\begin{prop}\label{prop:RichDim}
Let $\bfO\subset \fg^*$ be a Richardson nilpotent orbit that lies in the image of $\mu$. \DimaE{Then  every component of $\mu^{-1}(\bfO)$ has dimension at least $\dim \bfX+\dim \bfO/2$. In particular, $c_{\bfO}(\bfX)\geq 0$.}
\end{prop}
\begin{proof}
Let $\bfP\subset \bfG$ be a parabolic subgroup such that $\bfO=\bfO_\bfP$.
As in the previous theorem, let  $\bfS:=\mu^{-1}(\fp^{\bot})\sub T^*\bfX$. Then $\bfS$ is  a union of conormal bundles to $\bfP$-orbits on $\bfX$ and thus the dimension of every irreducible component of $\bfS$ is at least $\dim \bfX$.
\DimaE{
 Let $\bfS^0:=\mu^{-1}(\bfO\cap \fp^{\bot})$. Since $\bf O=O_P$, the intersection $\bfO\cap \fp^{\bot}$ is open in $\fp^{\bot}$. Thus $\bfS^0$ is open in $\bfS$ and thus the dimension of every irreducible component of $\bfS^{0}$ is at least $\dim \bfX$.
 Denote the irreducible components of $\bfS^0$ by $\bfY_i$. Let $\varphi\in \bfO\cap \fp^{\bot}$, and let $\bfY_{ij}$ be the irreducible components of $\mu^{-1}(\{\varphi\})\cap \bfY_i$. We have
$\mu^{-1}(\{\varphi\})=\bigcup_{ij}\bfY_{ij}$. By \cite[Theorem I.8.2]{Mum}, $$\dim \bfY_{ij}\geq \dim \bfY_i-\dim \bfO\cap \fp^{\bot}\geq \dim \bfX - \dim\bfO\cap \fp^{\bot}.$$
Thus every component of $\mu^{-1}(\{\varphi\})$ has dimension at least $\dim \bfX- \dim\bfO\cap \fp^{\bot}.$
Since the fibers under $\mu$ of all points in $\bfO$ are isomorphic, every irreducible component of $\mu^{-1}(\bfO)$ has dimension at least $$\dim \bfX- \dim\bfO\cap \fp^{\bot}+\dim \bfO.$$

 \DimaB{By Proposition \ref{prop:Piso} we have  $\dim \bfO\cap \fp^{\bot}\leq\dim \bfO/2.$}
Altogether, we \DimaC{get that every irreducible component of $\mu^{-1}(\bfO)$ has dimension at least
$$\dim \bfX- \dim\bfO\cap \fp^{\bot}+\dim \bfO \geq \dim \bfX+\dim \bfO/2.$$}}
\end{proof}

\DimaF{Let us now derive the following corollary, that is equivalent to Theorem \ref{IntThm:Rich}.}
\DimaB{
\begin{cor}\label{cor:Rich}
Let $\bf P\sub G$ be a parabolic subgroup, and let $\bf O\sub \overline{O_{P}}$ be a Richardson orbit. \DimaE{Then the intersection $\bfO\cap \fp^{\bot}$ is Lagrangian in $\bfO$, and in particular $c_{\bf O}(\bf G/P)=0$.}
\end{cor}
\begin{proof}
\DimaE{
By Proposition \ref{prop:Piso}, the intersection $\bfO\cap \fp^{\bot}$ is isotropic. Thus, it is enough to prove that the dimension of every component of $\bfO\cap \fp^{\bot}$ is at least $\dim \bfO/2$.
By Lemma \ref{lem:inO}, this statement follows from  Proposition \ref{prop:RichDim} applied to $\bf X':=G/P$.
}
\end{proof}

The corollary is especially useful for $\bfG$ of type A, since in this type all nilpotent orbits are Richardson. The assumption that $\bfO$ is Richardson cannot be omitted in the corollary, and thus also not in  Proposition \ref{prop:RichDim}, as the following example shows.
\begin{example}[D.I. Panyushev]\label{ex:small}
Let $\bf G:=\Sp_{2n}$, and let $\bf P\sub G$ be the stabilizer of a line. Let $\bfO:=\bfO_{\min}\subset \fg^*$ denote the minimal orbit, which under the trace form is identified with rank one matrices. Then $\dim \bfO_{\min}\cap \fp^{\bot}=1$, while $\dim \bfO_{\min}=2n$. \DimaC{Thus $c_{\bfO_{\min}}(\bf G/P)<0$ for $n>1$.}
\end{example}
}

\begin{lem}\label{lem:NonRed}
The subgroup $\bf H \sub G$ is $\bfO$-spherical if and only if the subgroup $\bf HU_G\sub G$ is $\bfO$-spherical.
\end{lem}
\begin{proof}
By \DimaE{Corollary \ref{cor:invO}} and using the fact that $\bfO\sub \fu_{\fg}^{\bot}$ we have
\begin{multline*}
{\bf H\sub G} \text{ is }\bfO\text{-spherical } \iff \dim \bfO\cap \fh^{\bot}\leq \dim \bfO/2\\\iff  \dim \bfO\cap (\fh+\fu)^{\bot}\leq \dim \bfO/2\iff {\bf HU_G\sub G} \text{ is }\bfO\text{-spherical.}
\end{multline*}
\end{proof}

\begin{lem}\label{lem:indSpher}
Let $\Xi\sub \fr_{\fg}^*$ be a closed conical subset, and let $\bf X$ be a $\Xi$-spherical algebraic $\bf R_G$-manifold. Embed $\bf R_G$ into $\bf G$ and consider the induced $\bf G$-manifold $\bf G\times_{\bf R_G}X$.
Then $\bf G\times_{\bf R_G}X$ is also $\Xi$-spherical.
\end{lem}
\begin{proof}
For any $\varphi\in \fu^{\bot}_{\fg}\sub \fg^*$, the fiber of $\varphi$ under the moment map $\mu_{\bf G\times_{\bf R_G}X}$ is isomorphic to $\bf G\times_{\bf R_G}\mu_X^{-1}(\varphi)$. Thus
\DimaG{for any orbit $\bfO\sub \Xi$ we have
 $$\bf \dim \mu_{\bf G\times_{\bf R_G}X}^{-1}(\bfO)=\dim U_G + \dim \mu_{R_G}^{-1}(\bfO)\leq \dim U_G+\dim X+\dim \bfO/2=\dim G\times_{\bf R_G}X+\dim \bfO/2$$
}
\end{proof}
\DimaH{
\subsection{Symmetric spaces}\label{subsec:GeoYam}
We have $\cK_G\cong G/K$, where $K\sub G$ is a maximal compact subgroup. Furthermore, $K$ is the group of real points of some symmetric subgroup $\bf K\sub G$ defined over $\R$. Thus $\cK_G$ is a connected component of the set of real points of $\bf Y:=G/K$.

\begin{proof}[Proof of Proposition \ref{intprop:GeoYam}]
If $\Xi=\{0\}$ the statement is obvious. Otherwise let $\bfO\sub \Xi$ be a non-zero $\bfG$-orbit. Let $ \fk$ denote the Lie algebra of $\bf K$. 
Let ${\bf S_O}\sub T^*({\bf G/H})$ be the preimage $${\bf S_O}:=\mu^{-1}(\bfO\cap \fk^{\bot}).$$ 
Let $p_{\bf G/H}:T^*{\bf G/H}\to {\bf G/H}$ denote the natural projection. Let us first show that $p_{\bf G/H}({\bf S_O})$ is not dense in $\bf G/H$. Indeed, by the assumption we have
$$\dim \mu^{-1}(\bfO)\leq \dim {\bf G/H}+\dim \bfO/2$$
Since $\mu$ is a $\bfG$-equivariant map, the fibers of all the points in $\bfO$ are isomorphic, and thus have the same dimension. This dimension is thus
$$\dim \mu^{-1}(\bfO)-\dim \bfO\leq \dim {\bf G/H}+\dim \bfO/2-\dim \bfO=\dim {\bf G/H}-\dim \bfO/2$$
By \cite[Proposition 3.1.1]{Gin}, $\bfO \cap \fk^{\bot}$ is either empty, or has dimension $\dim \bfO/2$.
Thus, $\bf S_O$ is either empty, or has dimension 
$$\dim {\bf S_O}=\dim \mu^{-1}(\bfO\cap \fk^{\bot})=\dim \mu^{-1}(\bfO)-\dim \bfO+\dim \bfO/2\leq \dim {\bf G/H}.$$
Thus, if the restriction of projection $p_{\bf G/H}$ to $\bf S_O$ is dominant then its generic fiber is finite. However, the fibers are conical and non-zero, and thus cannot be finite. Thus, the projection is not dominant and $p_{\bf G/H}({\bf S_O})$ is not dense, for any non-zero orbit $\bfO\sub \Xi$.

Now, for every $g\in \bfG$, we have $\mu(p_{\bf G/H}^{-1}(g\bfH))=g\cdot\fh^{\bot}$. Thus, if $g\bfH$ does not lie in the image of $p_{\bf G/H}(\bf S_O)$ then $g\cdot\fh^{\bot}\cap\bfO\cap \fk^{\bot}=\emptyset$, and thus $\fh^{\bot}\cap\bfO\cap g^{-1}\cdot\fk^{\bot}=\emptyset$. Let $\bf U'\sub G/H$ denote the complement to the union of the closures of all $p_{\bf G/H}({\bf S_O})$ for all non-zero orbits $\bfO\sub \Xi$, and let $\bf U''\sub G$ be the image of the preimage of $\bf U'$ in $\bfG$ under the inversion map $g\mapsto g^{-1}$. Then for every $g\in {\bf U''}$ and every non-zero $\bfO\sub \Xi$ we have $\fh^{\bot}\cap\bfO\cap g\cdot\fk^{\bot}=\emptyset$.
Thus for every $g\in {\bf U''}$ we have $\fh^{\bot}\cap\Xi\cap g\cdot\fk^{\bot}=\{0\}$.
Thus we take $\bf U:=U''/K\sub Y$, and $U$ to be the intersection of $\cK_G$\ with the set of real points of $\bf U$.
\end{proof}
}
\subsection{Branching problems}
Let $\bf H\subset G$ be an algebraic  subgroup, and consider the subgroup $\bf \Delta H \subset G\times H$. Let $\bfO_1\subset \fg^*$ and $\bfO_2\subset \fh^*$ be nilpotent orbits.
In this subsection we \DimaC{study the $\overline{\bfO_1}\times\overline{\bfO_2}$- complexity of the  $\bf G\times H$-space
$\bf G\times H/\Delta H \cong G$.

Theorem \ref{IntThm:Ri} gives the following immediate corollary.
\begin{cor}\label{cor:RichBran}
If $\bfO_1=\bfO_\bfP$ for some parabolic subgroup $\bf P\subset G$  and $\bfO_2=\bfO_{\bfQ}$ for some parabolic subgroup $\bf Q\subset H$ then the following are equivalent
\begin{enumerate}[(i)]
\item $\bf G$ is $\overline{\bfO_1}\times\overline{\bfO_2}$-spherical.
\item The set of double cosets $\bf P\backslash G /Q$ is finite.
\item $\bf G/P$ is $\overline{\bfO_2}$-spherical as an $\bf H$-space.
\item $\bf G/Q$ is $\overline{\bfO_1}$-spherical as a $\bf G$-space.
\end{enumerate}
\end{cor}

We would now like to obtain a version of this corollary in the case when only one of the orbits is Richardson.

\begin{prop}\label{prop:BranchComp}$\,$
\begin{enumerate}[(i)]
\item \label{it:cO1}If $\bfO_1=\bfO_{\bf P}$ for some parabolic subgroup $\bf P\subset G$ then
$$c_{\overline{\bfO_2}}({\bf G/P})\leq c_{\overline{\bfO_1}\times\overline{\bfO_2}}(\bfG)\leq c_{\overline{\bfO_2}}({\bf G/P})-\min_{\bfO_1'\sub \overline{\bfO_1}}c_{\bfO_1'}({\bf G/P})$$
\item \label{it:cO2} If $\bfO_2=\bfO_{\bf Q}$ for some parabolic subgroup $\bf Q\subset H$ then
$$c_{\overline{\bfO_1}}({\bf G/Q})\leq c_{\overline{\bfO_1}\times\overline{\bfO_2}}(\bfG)\leq c_{\overline{\bfO_1}}({\bf G/Q})-\min_{\bfO_2'\sub \overline{\bfO_2}}c_{\bfO_2'}({\bf H/Q})$$
\end{enumerate}
\end{prop}
\begin{proof}
For any two nilpotent orbits $\bfO'_1\susbet \overline{\bfO_1}$ and $\bfO'_2\subset \overline{\bfO_2}$, and denote $$\Upsilon(\bfO_1',\bfO_2'):=\{(\varphi,\psi)\in \bfO'_1\times \bfO'_2\, \vert \varphi|_\fh = \psi\}$$

By \DimaE{Corollary \ref{cor:invO}} applied to $\bf G'=G\times H$ and $\bf H'=\Delta H\sub G'$,
$$c_{\bf O_1'\times O_2'}(\bfG)=\dim \Upsilon(\bfO_1',\bfO_2') - \dim \bfO_1'/2- \dim \bfO_2'/2$$
\begin{enumerate}[Proof of (i)]
\item \label{it:PfBranchP} Denote by $\nu_1$ the moment map $T^*({\bf G/P})\to \fg^*$ and by $\nu_2$ the moment map  $T^*({\bf G/P})\to \fh^*$, which also equals the composition of $\nu_1$ with the restriction $\fg^*\to \fh^*$. Denote
$$\de(\bfO_1',\bfO_2'):=\dim\nu_1^{-1}(\bfO_1')\cap \nu_2^{-1}(\bfO_2')-\dim {\bf G/P}-\dim \bfO_2'/2$$
Since the image of $\nu_1$ equals $\overline{\bfO_1}$, we have
$$c_{\bfO'_2}({\bf G/P})=\max_{\bfO_1'\sub \overline{\bfO_1}}\de(\bfO_1',\bfO_2')$$
Denote
$${\bf S_P}:=\{(g,(\varphi,\psi))\in \bfG \times \Upsilon(\bfO_1',\bfO_2')\, \vert \, g^{-1}\cdot \varphi\in {\fp^{\bot}}\}.$$
Consider the natural projection ${\bf S_P}\to \Upsilon(\bfO_1',\bfO_2')$. The dimension of the fiber of any $(\varphi,\psi)\in \Upsilon(\bfO_1',\bfO_2')$ equals
$$\dim \bfG_{\varphi}+\dim \bfO_1'\cap \fp^{\bot}=\dim \bfG-\dim \bfO_1'+\dim \bfO_1'\cap \fp^{\bot}=\dim \bfG +c_{\bfO_1'}({\bf G/P})- \dim \bfO_1'/2$$
Thus we have
$$\dim {\bf S_P} = \dim \Upsilon(\bfO_1',\bfO_2') + \dim \bfG +c_{\bfO_1'}({\bf G/P})- \dim \bfO_1'/2$$

We also have a natural map $\bf S_P\to \nu_1^{-1}(\bfO_1')\cap \nu_2^{-1}(\bfO_2')$ by
$$(g,(\varphi,\psi))\mapsto (g\bfP,\varphi)$$
The dimension of every fiber is $\dim \bfP$, and thus
$$\dim {\bf S_P} = \dim \nu_1^{-1}(\bfO_1')\cap \nu_2^{-1}(\bfO_2') +\dim \bfP.$$
Altogether we get
\begin{multline*}
c_{\bf O_1'\times O_2'}(\bfG)=\dim \Upsilon(\bfO_1',\bfO_2') - \dim \bfO_1'/2- \dim \bfO_2'/2=\\ \dim {\bf S_P}-(\dim \bfG +c_{\bfO_1'}({\bf G/P})- \dim \bfO_1'/2)- \dim \bfO_1'/2-\dim \bfO_2'/2=\\
 \dim \nu_1^{-1}(\bfO_1')\cap \nu_2^{-1}(\bfO_2') +\dim \bfP -(\dim \bfG +c_{\bfO_1'}({\bf G/P})- \dim \bfO_1'/2)- \dim \bfO_1'/2-\dim \bfO_2'/2=\\
\dim \nu_1^{-1}(\bfO_1')\cap \nu_2^{-1}(\bfO_2') -\dim {\bf G/P} -c_{\bfO_1'}({\bf G/P})-\dim \bfO_2'/2=\de(\bfO_1',\bfO_2')-c_{\bfO_1'}({\bf G/P})
\end{multline*}
Thus we have
$$c_{\bf O_1'\times O_2'}(\bfG)+c_{\bfO_1'}({\bf G/P})= \de(\bfO_1',\bfO_2')$$
Thus
\begin{multline*}
\max_{\bfO_1'\sub \overline{\bfO_1},{\bfO_2'\sub \overline{\bfO_2}}} c_{\bf O_1'\times O_2'}(\bfG) + \min_{\bfO_1'\sub \overline{\bfO_1}} c_{\bfO_1'}({\bf G/P})  \leq \max_{\bfO_1'\sub \overline{\bfO_1},{\bfO_2'\sub \overline{\bfO_2}}}\de(\bfO_1',\bfO_2')\leq\\ \max_{\bfO_1'\sub \overline{\bfO_1},{\bfO_2'\sub \overline{\bfO_2}}} c_{\bf O_1'\times O_2'}(\bfG) + \max_{\bfO_1'\sub \overline{\bfO_1}} c_{\bfO_1'}({\bf G/P})
\end{multline*}
\DimaG{We have $\max_{\bfO_1'\sub \overline{\bfO_1}} c_{\bfO_1'}({\bf G/P})=0$ since $\bf G/P$ is a spherical $\bfG$-space.}
Thus
$$c_{\overline{O_1}\times \overline{O_2}}(\bfG)+ \min_{\bfO_1'\sub \overline{\bfO_1}} c_{\bfO_1'}({\bf G/P}) \leq c_{\overline{O_2}}({\bf G/P})\leq c_{\overline{O_1}\times \overline{O_2}}(\bfG)$$
This implies the assertion.

\item We proceed in a similar way. Denote by $\mu_1$ the moment map $T^*({\bf G/Q})\to \fg^*$.
\DimaD{Denote by $\fq^{\bot}$ the orthogonal complement (annihilator) to $\fq$ in $\fg^*$, and by $\fq^{\bot}_{\fh}$ the restriction of $\fq^{\bot}$ to $\fh^*$, which also equals the orthogonal complement to $\fq$ in $\fh^*$.} Consider the natural map
$$\eta:\bfG \times \fq^{\bot}\to {T^*}({\bf G/Q}) \text{ given by }\eta(g,\varphi):=(g\bfQ,g\cdot \varphi),$$
and let ${\bf T}:=\eta^{-1}(\mu^{-1}(\bfO'_1))$. Then $\dim {\bf T}=\dim \mu^{-1}(\bfO'_1)+\dim \bfQ$.  Denote also
$${\bf T}(\bfO_1',\bfO_2'):=\{(g,\varphi)\in {\bf T}\, \vert \, \varphi|_{\fh}\in \bfO'_2 \}$$
\DimaD{For any $(g,\varphi)\in {\bf T}$ we have $\varphi|_{\fh}\in \fq^{\bot}_{\fh}\sub \overline{\bfO_2},$  thus} $\bf T=\bigcup_{\bfO'_2\sub \overline{\bfO_2}} {\bf T}(\bfO_1',\bfO_2')$.  Denote
$$d(\bfO_1',\bfO_2'):=\dim {\bf T}(\bfO_1',\bfO_2')-\dim {\bf G}-\dim \bfO_1'/2$$
Since $\dim {\bf T}=\dim \mu^{-1}(\bfO'_1)+\dim \bfQ$ and $\bf T=\bigcup_{\bfO'_2\sub \overline{\bfO_2}} {\bf T}(\bfO_1',\bfO_2')$ we have
$$c_{\bfO'_1}({\bf G/Q})=\max_{\bfO_2'\sub \overline{\bfO_2}}d(\bfO_1',\bfO_2')$$
Thus, \DimaD{as in the proof of \eqref{it:PfBranchP}, }it is enough to show that
\begin{equation}\label{=cQ}
d(\bfO_1',\bfO_2')-c_{\bfO'_2}({\bf H/Q})=c_{\bfO_1'\times \bfO_2'}(\bfG)
\end{equation}
For this purpose consider the map
$$\zeta:{\bf T}(\bfO_1',\bfO_2')\to \Upsilon(\bfO_1',\bfO_2'), \quad (g,\varphi)\mapsto (\varphi,\varphi|_{\fh}) .$$
The image is $\Upsilon(\bfO_1',\bfO_2')\cap \bfO_1'\times (\bfO_2'\cap \fq_{\fh}^{\bot})$.
Since the diagonal action of $\bfH$ preserves $\Upsilon(\bfO_1',\bfO_2')$, and commutes with the projection $\Upsilon(\bfO_1',\bfO_2')\onto \bfO_2'$,
\DimaD{all the fibers of this projection are isomorphic. Thus
 we have }
\begin{multline*}
\dim \Upsilon(\bfO_1',\bfO_2')\cap \bfO_1'\times (\bfO_2'\cap \fq_{\fh}^{\bot}) =\dim \Upsilon(\bfO_1',\bfO_2')-\dim \bfO_2'+\dim \bfO_2'\cap \fq_{\fh}^{\bot}=\\
\dim \Upsilon(\bfO_1',\bfO_2')+c_{\bfO'_2}({\bf H/Q})-\dim \bfO_2'/2.
\end{multline*}

Now, every non-empty fiber $\zeta^{-1}(\varphi,\psi)$ is isomorphic to $\bfG$. Thus we have
$$\dim {\bf T}(\bfO_1',\bfO_2') = \dim \Upsilon(\bfO_1',\bfO_2') + c_{\bfO'_2}({\bf H/Q})-\dim \bfO_2'/2+\dim \bfG ,$$
and
\begin{multline*}
d(\bfO_1',\bfO_2')=\dim \Upsilon(\bfO_1',\bfO_2') + c_{\bfO'_2}({\bf H/Q})-\dim \bfO_2'/2+\dim \bfG -\dim {\bf G}-\dim \bfO_1'/2=\\ c_{\bfO_1'\times \bfO_2'}(\bfG)+c_{\bfO'_2}({\bf H/Q})
\end{multline*}

This implies \eqref{=cQ}, which in turn implies the assertion.
\end{enumerate}
\end{proof}

\begin{cor}\label{cor:BranchSpher}$\,$
\begin{enumerate}[(i)]
\item \label{it:O1} If $\bfO_1=\bfO_\bfP$ for some parabolic subgroup $\bf P\subset G$ and $\bf G$ is $\overline{\bfO_1}\times\overline{\bfO_2}$-spherical then $\bf G/P$ is an $\overline{\bfO_2}$-spherical $\bfH$-space.

\item \label{it:O1iff} If $\bfO_1=\bfO_\bfP$ for some parabolic subgroup $\bf P\subset G$ and either all orbits $\bfO_1'\sub \overline{\bfO_1}$ are Richardson, or $\bf P$ is a Borel subgroup
then\\ $\bf G$ is $\overline{\bfO_1}\times\overline{\bfO_2}$-spherical if and only if $\bf G/P$ is an $\overline{\bfO_2}$-spherical $\bfH$-space.

\item \label{it:O2} If $\bfO_2=\bfO_\bfQ$ for some parabolic subgroup $\bf Q\subset H$ and  $\bf G$ is $\overline{\bfO_1}\times\overline{\bfO_2}$-spherical then $\bf G/Q$ is an $\overline{\bfO_1}$-spherical $\bfG$-space.

\item \label{it:O2iff} If $\bfO_2=\bfO_\bfQ$ for some parabolic subgroup $\bf Q\subset H$ and either all orbits $\bfO_2'\sub \overline{\bfO_2}$ are Richardson, or $\bf Q$ is a Borel subgroup then\\ $\bf G$ is $\overline{\bfO_1}\times\overline{\bfO_2}$-spherical if and only if $\bf G/Q$ is an $\overline{\bfO_1}$-spherical $\bfG$-space.

\end{enumerate}
\end{cor}

\begin{proof}
Parts \eqref{it:O1} and \eqref{it:O2} follow immediately from Proposition \ref{prop:BranchComp}.
For Part \eqref{it:O1iff} we note that for any  Richardson orbit $\bfO_1'\sub \overline{\bfO_1},\, c_{\bfO_1'}({\bf G/P})=0$ by Corollary \ref{cor:Rich}, and if $\bf P$ is a Borel subgroup of $\bf G$ then for any orbit $\bfO_1'\sub \overline{\bfO_1},\, c_{\bfO_1'}({\bf G/P})=0$ by Proposition \ref{prop:Piso}. Thus \eqref{it:O1iff}  also follows from Proposition \ref{prop:BranchComp}.  The same arguments prove part \eqref{it:O2iff}.
\end{proof}
}

\section{Invariant distributions and D-modules}\label{sec:Dmod}


In sections \ref{sec:Dmod} and \ref{sec:main} all the algebraic varieties and algebraic groups we will consider will be defined over $\R$. We will use boldface letters like $\bfX, \bfG$ to denote these algebraic varieties and groups, and the corresponding letters in regular font (like $X, G$) to denote their real points. The Gothic letters (like $\fg$) will  denote the complexified Lie algebras.

\subsection{Preliminaries and notation}

\subsubsection{D-modules}

We will use the theory of D-modules on complex algebraic manifolds.
We will now recall some facts and notions that we will use.
For a good introduction to the algebraic theory of D-modules, we refer the reader to \cite{Ber} and \cite{Bor}. For a short overview, see \cite[Appendix B]{AMOT}.
By a \emph{D-module} on a smooth algebraic variety $\bfX$, we mean a quasi-coherent sheaf of right modules over the sheaf $D_\bfX$ of algebras of algebraic differential operators.  By a \emph{finitely generated D-module} on a smooth algebraic variety $\bfX$ we mean a coherent sheaf of right modules over the sheaf $D_\bfX$.
Denote the category of $D_\bfX$-modules by $\cM(D_\bfX)$.

For a smooth affine variety $\bf V$, we denote $D({\bf V}):=D_{\bf V}({\bf V})$.
Note that the category $\cM(D_{\bf V})$ of D-modules on ${\bf V}$ is equivalent to the category of $D({\bf V})$-modules.
We will thus identify these categories.

The algebra $D({\bf V})$ is equipped with a filtration which is called the geometric filtration and defined by the degree of differential operators. The associated graded algebra with respect to this filtration is the algebra $\cO(T^*{\bf V})$ of regular functions on the total space of the cotangent bundle of ${\bf V}$. This allows us to define the \emph{singular support} of a finitely generated D-module $M$ on ${\bf V}$ in the following way. Choose a good filtration on $M$, i.e., a filtration such that the associated graded module is a finitely-generated module over $\cO(T^*{\bf V})$, and define the singular support $SS(M)$ to be the support of this module. One can show that the singular support does not depend on the choice of a good filtration on $M$. By Bernstein's inequality, $\dim SS(M)\geq \dim {\bf V}$. If for every $x\in {\bf V}$ we have $\dim_x SS(M)= \dim_x\mathbf{V}$ the module $M$ is called \emph{holonomic}.

\subsubsection{Nuclear \Fre $\,$spaces}\label{subsubsec:NF}
For nuclear \Fre $\,$spaces $V$ and $W$, $V\otimes W$ will denote the completed projective tensor product and $V^*$  will denote the continuous linear dual, endowed with the strong dual topology
(see e.g. \cite[\S 50]{Tre}).
\begin{lem}[{\cite[(50.18) and (50.19]{Tre}}]\label{lem:NF}
$(V\otimes W)^*\cong V^*\otimes W^*$ and $L(V,W)\cong V^*\otimes W$.
\end{lem}

\subsubsection{Schwartz functions and tempered distributions}
We will use the theory of Schwartz functions and tempered distributions on Nash  manifolds, see e.g., \cite{dCl,AG_Sc}.
Nash manifolds are smooth semi-algebraic manifolds, e.g. real points of algebraic manifolds defined over $\R$. This is the only type of Nash manifolds that appears in the current paper.

We denote the space of complex valued Schwartz functions on a Nash manifold $Y$ by $\Sc(Y)$. For a Nash bundle $\cE$ on $Y$, we denote its space of complexified  Schwartz sections  by $\Sc(X,\cE)$.
We denote by $\Sc^*(Y)$ the (continuous) dual space to $\Sc(Y)$, and call its elements \emph{tempered distributions}. Similarly, we denote
$\Sc^*(Y,\cE):=(\Sc(Y,\cE))^*$.
If $\cE$ is a vector bundle over an algebraic manifold $\bfY$ defined over $\R$, we will also consider it as a Nash bundle over the Nash manifold $Y$ of real points of $\bfY$. For a Nash manifold $Y$ we denote by $\Sc^*_Y$ the sheaf of tempered distributions on $Y$ defined by $\Sc^*_Y(U):=\Sc^*(U)$. For an  algebraic manifold $\bfY$ defined over $\R$, we will denote by $\Sc^*_{\bf Y}$ the quasi-coherent sheaf defined by $\Sc^*_{\bf Y}({\bf U}):=\Sc^*(U)$. This is a sheaf of $\cD_{\bfY}$-modules. A distribution in $\Sc^*(Y)$ is called holonomic if it generates a holonomic $\cD_{\bfY}$-module. Since this notion is local and is invariant with respect to linear transformations, it naturally extends to elements of $\Sc^*(Y,\cE)$.



\begin{theorem}[{\rm{Bernstein--Kashiwara}, see {\it e.g.} \cite[Theorem 3.13]{AGM}, cf. \cite[Theorems 5.1.7 and 5.1.12]{Kas}}]\label{thm:DimSol}
Let $\bfY$ be an algebraic manifold defined over $\R$, and let $M$ be a holonomic $D_{\bf Y}$-module. Then the space of solutions $\Hom_{D_{\bf Y}}(M,\Sc^*_{\bfY})$
  is finite-dimensional.
\end{theorem}

We will also need the following version of the Schwartz Kernel theorem.
\begin{lem}[{See {\it e.g.} \cite[Corollary 2.6.3]{AG_RS}}]\label{lem:SKT}
Let $ X$ and $Y$ be Nash manifolds, and let $\cE_1$ and $\cE_2$ be Nash bundles on them. Then
$$\Sc(X\times Y, \cE_1\boxtimes \cE_2)\cong \Sc(X,\cE_1)\otimes \Sc(Y,\cE_2)$$
\end{lem}

\subsubsection{Notation}\label{subsec:not}
Let $X$ be the manifold of real points of an algebraic $\bfG$-manifold $\bfX$.
We will say that a $G$-equivariant bundle $\cE$ on $X$ is a \emph{twisted algebraic} $G$-bundle if there exists a finite-dimensional (smooth) representation $\sigma$ of $G$ and an algebraic $G$-bundle $\cE'$ on $X$ such that $\cE=\cE'\otimes \fin$.

For any \DimaG{left} ideal $I\subset \cU(\fg)$, denote by $\cV(I)$ its associated variety, i.e. the closed conical subset of $\fg^*$ defined to be the set of zeros of the symbols (in $S(\fg)$) of the elements of $I$. If $I$ intersects the center $\fz$ of $\cU(\fg)$ by an ideal of finite codimension then $\cV(I)$ is known to be a union of nilpotent orbits.

\subsection{Holonomicity and finite-dimensionality of spaces of invariant distributions}

\begin{thm}\label{Thm:Dist}
Let $I\subset \cU(\fg)$ be a two-sided ideal, and let $\cV(I)\subset \fg^*$ denote its associated variety. Suppose that  $\cV(I)$ is a union of finitely many $\bf G$-orbits.
Let $\bf X,Y$ be $\cV(I)$-spherical $\bfG$-manifolds.
Let $\cE_1, \cE_2$ be twisted algebraic $G$-bundles on $X$ and $Y$ respectively.

Let $\Sc^*(X\times Y,\cE_1\boxtimes \cE_2)^{\Delta \fg,I}$ denote the subspace of $\Sc^*(X\times Y,\cE_1\boxtimes \cE_2)$ consisting of elements invariant under the diagonal action of $\fg$, and annihilated by the action of $I$ on the first coordinate.
Then
$\Sc^*(X\times Y,\cE_1\boxtimes \cE_2)^{\Delta \fg,I}$ is finite-dimensional, and consists of holonomic distributions.
\end{thm}
For the proof we will need the following straightforward statements.
\begin{lem}\label{lem:gbundle}
Let $\bf X$ be an algebraic manifold.  Let $\cE$ be a vector bundle over $\bf X$, and let $\cO_{\bfX,\cE}$ be the locally free coherent sheaf of regular sections of $\cE$. Then we have canonical isomorphisms
$$\Hom_{\cD_{\bfX}}(\cO_{\bfX,\cE}\otimes_{\cO_{\bfX}} \cD_{\bfX}, \Sc^*_{\bfX})\cong \Hom_{\cO_{\bfX}}(\cO_{\bfX,\cE}, \Sc^*_{\bfX})\cong\Sc^*(X,\cE)$$
such that
\begin{enumerate}[(i)]
\item these isomorphisms are functorial on the groupoid of pairs $(\bfX,\cE).$
\item \label{it:fact} If a map $\varphi\in \Hom_{\cD_{\bfX}}(\cO_{\bfX,\cE}\otimes_{\cO_{\bfX}} \cD_{\bfX}, \Sc^*_{\bfX})$ factors through a holonomic $\cD_{\bfX}$-module then the corresponding distribution in $\Sc^*(X,\cE)$ is holonomic.
\end{enumerate}
\end{lem}
\DimaG{
\begin{proof}
Define maps $$\varphi:\Hom_{\cD_{\bfX}}(\cO_{\bfX,\cE}\otimes_{\cO_{\bfX}} \cD_{\bfX}, \Sc^*_{\bfX})\to \Hom_{\cO_{\bfX}}(\cO_{\bfX,\cE}, \Sc^*_{\bfX})$$ and 
$$\psi: \Sc^*(X,\cE)
\to \Hom_{\cO_{\bfX}}(\cO_{\bfX,\cE}, \Sc^*_{\bfX})$$
by $\varphi(\alp)(\nu):=\alp(\nu\otimes 1)$ and $\langle \psi(\xi)(\nu), \rho\rangle:=\langle \xi|_{\bfU(\R)},\nu\rho\rangle,$ where 
$$\bf U\sub X \text{ is a Zariski open set, }\nu\in \cO_{\bfX,\cE}(\bfU), \quad \alp\in \Hom_{\cD_{\bfX}}(\cO_{\bfX,\cE}\otimes_{\cO_{\bfX}} \cD_{\bfX}, \Sc^*_{\bfX}), \quad \rho \in \Sc(\bfU(\R))$$
The fact that $\varphi$ and $\psi$ are isomorphisms, as well as part \eqref{it:fact} of the lemma can be checked locally on $\bfX$. This allows to reduce to the case when $\bf X$ is affine and $\cE$ is trivial. This case is straightforward.
\end{proof}
}

\begin{cor}\label{cor:gbundle}
Let an algebraic group $\bfG$ act on an algebraic manifold $\bf X$. Let  $\cE$ be a twisted algebraic $G$-bundle on $X$. Then we have an isomorphism of $\fg$-modules
$$\Hom_{\cD_{\bfX}}(\cO_{\bfX,\cE}\otimes_{\cO_{\bfX}} \cD_{\bfX}, \Sc^*_{\bfX})\cong\Sc^*(X,\cE).$$
\end{cor}

\begin{proof}[Proof of Theorem \ref{Thm:Dist}]

Consider $N:=\cO_{X\times Y,\cE_1\boxtimes \cE_2}\otimes_{\cO_{\bf X \times Y}}\cD_{\bf X \times Y}$ as a $\cD_{\bf X \times Y}$-module equipped with an action of $\fg\times \fg$, and thus an action of $\cU(\fg)\otimes_{\C}\cU(\fg)$.  Let $J\sub \cU(\fg)\otimes_{\C}\cU(\fg)$ be the ideal generated by $\Delta \fg$ and by $I\otimes 1$.  Let $M:=N/JN$. The right action of $\cD_{\bf X\times Y}$ on itself defines $\cD_{\bf X\times Y}$-module structures on $N$ and on $M$.

By Corollary \ref{cor:gbundle}, the linear space $\Sc^*(X\times Y,\cE_1\boxtimes\cE_2)^{\Delta \fg,I}$ is  isomorphic to the  space of solutions  $\Hom_{\cD_{\bf X\times Y}}(M,\Sc^*_{\bf X\times Y})$.  The singular support of $M$ lies in the preimage  of $$\bfR=(\cV(I)\times \cV(I))\cap (\Delta\fg)^{\bot}$$ under the moment map $\mu_{X\times Y}:T^*{\bf X}\times T^*\bfY\to \fg^*\times \fg^*$. For every point $(x,-x)\in \bfR$ with orbit $\bfO\times \bfO$ we have $$\dim \mu^{-1}(\{(x,-x)\})=\dim \mu_X^{-1}(\bfO)+\dim \mu_Y^{-1}(\bfO)-2\dim \bfO\leq \dim \bfX+\dim \bfY -\dim \bfO.$$

Thus, $\dim \mu^{-1}((\bfO\times \bfO) \cap \Delta\fg^{\bot})\leq  \dim \bfX \times  \bfY $.
Since $\bf R$ is a finite union of sets of the form $(\bfO\times \bfO) \cap \Delta\fg^{\bot}$, we get that $\dim \mu^{-1}(\bfR)\leq \dim \bfX \times  \bfY$, and thus $M$ is holonomic. By Theorem \ref{thm:DimSol}, this implies that the space of solutions is finite-dimensional. By Lemma \ref{lem:gbundle}, the space $\Sc^*(X\times Y,\cE_1\boxtimes \cE_2)^{\Delta \fg,I}$ consists of holonomic distributions.
\end{proof}

\section{Representation theory}\label{sec:main}

Throughout the section we fix a connected linear algebraic group $\bf G$ defined over $\R$.

\subsection{Preliminaries}
\DimaI{
Let $\widetilde{G}$ be a finite cover of an open subgroup of $G$.}
We denote by $\Rep^{\infty}(\widetilde{G})$ the category of smooth nuclear \Fre\  representations of $\widetilde{G}$ of moderate growth. This is essentially the same definition as in \cite[\S 1.4]{dCl} with the additional assumption that the representation spaces are nuclear (see \S \ref{subsubsec:NF}). For example, for any algebraic $\bfG$-manifold $\bfX$ and any twisted algebraic $G$-bundle $\cE$ over $X$, the representation $\Sc(X,\cE)$ lies in $\Rep^{\infty}(G)$.

For any two-sided ideal $I\subset \cU(\fg)$, and any representation $\Pi\in \Rep^{\infty}(\widetilde{G})$, denote by $\Pi_I\in \Rep^{\infty}(\widetilde{G})$ the representation $\Pi/\overline{I\Pi}$, where $\overline{I\Pi}$ denotes the closure of the action of $I$.

\DimaH{For any $\fg$-module $V$, denote by $\An V\sub \cU(\fg)$ its annihilator ideal.}

For any \Fre $\,$space $V$, \cite[\S 1.2]{dCl} defines the space of $V$-valued Schwartz functions $\Sc(Y,V)$, and shows that the natural map $\Sc(Y)\otimes V \to \Sc(Y,V)$ is an isomorphism.

\begin{defn}\label{def:ind}
For a closed semi-algebraic subgroup $H\subset \widetilde{G}$ and $\pi\in \Rep^{\infty}(H),$ we denote by $\ind_H^{\widetilde{G}}(\pi)$ the Schwartz induction as in \cite[\S 2]{dCl}. More precisely, in \cite{dCl} du Cloux considers the space  $\Sc(\widetilde{G},\pi)$ of Schwartz functions from $G$ to the underlying space of $\pi,$ and defines a map from  $\Sc(\widetilde{G},\pi)$  to the space $C^{\infty}(\widetilde{G},\pi)$ of all smooth $\pi$-valued functions on $\widetilde{G}$ by $f\mapsto \overline f$, where \begin{equation}\label{=ind}
\overline f (x)=\int_{h\in H}\pi(h)f(xh)dh,
\end{equation}
and $dh$ denotes a fixed left-invariant measure on $H$. The Schwartz induction $\ind_H^G(\pi)$ is defined to be the image of this map. Note that $\ind_H^G(\pi)\in \Rep^{\infty}(\widetilde{G})$.
\end{defn}

%

\begin{defn}
For $\pi\in \Rep^{\infty}(\widetilde{G})$, denote by $\pi_G$ the space of coinvariants, i.e. quotient of $\pi$ by the intersection of kernels of all $G$-invariant functionals. Explicitly,
$$\pi_G=\pi/\overline{\{\pi(g)v -v\, \vert \,v\in \pi, \, g\in G\}}. $$
\end{defn}
Note that if $G$ is connected then $\pi_G=\pi/\overline{\fg\pi}$ which in turn is equal to the quotient of $\oH_0(\fg,\pi)$ by the closure of zero.

We will need the following two versions of the Frobenius reciprocity for Schwartz induction.

\begin{lem}[{\cite[Lemma 2.3.4]{GGS}}]\label{lem:Frob}
Let $\tau \in \Rep^{\infty}(H)$, $\pi\in \Rep^{\infty}(\widetilde{G})$ and let $\pi^*$ denote the dual representation, endowed with the strong dual topology.  Then
$$\Hom_{G}(\ind_H^{\widetilde{G}}(\tau),\pi^*) \cong \Hom_H(\tau,\pi^*\delta_H^{-1}\delta_G),$$
where $\delta_H$ and $\delta_G$ denote the modular functions of $H$ and $G$.
\end{lem}

\begin{lem}[{\cite[Lemma 2.8]{GGS2}}]\label{lem:Frob2} Let $\tau \in \Rep^{\infty}(H)$, $\pi\in \Rep^{\infty}(\widetilde{G})$.
Consider the diagonal actions of $H$ on $\pi\otimes \tau$ and of $G$ on $\pi \otimes \ind_H^{\widetilde{G}}(\tau)$. Then
$(\pi\otimes \tau\delta_H\delta_G^{-1})_{H}\cong(\pi \otimes \ind_H^{\widetilde{G}}(\tau))_{\widetilde{G}}$.
\end{lem}

\DimaI{
\subsubsection{Harish-Chandra modules}\label{sec:Yam}
In this subsubsection we fix an algebraic reductive group $\bf M$ defined over $\R$, a finite cover $\widetilde{M}$ of an open subgroup of ${\bf M}(\R)$, and a maximal compact subgroup $K\sub \widetilde{M}$.
Let $\fm$ denote the (complex) Lie algebra of $\bf M$, and $\fk$ denote the complexified Lie algebra of $K$.

\begin{defn}
An \emph{$(\fm, K)$-module} is an $\fm$-module with an additional structure of a representation of $K$ that is locally-finite, continuous\footnote{Note that for a locally-finite representation the notion of continuity does not require a topology on the vector space, since finite-dimensional complex vector spaces possess unique topology (that is compatible with the topology on $\C$).}, and compatible with the action of $\fm$ in the sense that the actions of $\fk$ obtained from $\fm$ (by restriction) and $K$ (by differentiation and complexification) coincide, and such that for every $k\in K$, every $\alp \in \fm$, and every $v\in M$ we have 
$$k(\alp v)=(\ad(k)\alp)kv,$$
where $\ad(k)\alp$ denotes the adjoint action of $K$ on $\fm$.
\end{defn}

\begin{defn}
An $(\fm, K)$-module $L$ is called a Harish-Chandra module if 
$L$  is finitely-generated over $\fm$, and 
for every irreducible  representation $\sigma$ of $K$ we have $$\dim\Hom_{K}(\sigma, L)<\infty.$$
We will denote the category of Harish-Chandra modules by $\cM(\fm,K)$.
\end{defn}

\begin{thm}[{\cite[\S4.2]{Wal}}]\label{thm:HCM}
Every Harish-Chandra module has  finite length. In particular, 
the intersection  $\An(L)\cap \fz(\cU(\fm))$ has finite codimension in the center $\fz(\cU(\fm))$ of $\cU(\fm)$.
\end{thm}



\begin{defn}
For any finitely-generated $\fm$-module $L$, define its \emph{associated variety} $\mathfrak{S}(L)\sub \fm^*$ to be $\cV(I)$, where $I\sub \cU(\fm)$ is the annihilator of a set of generators of $L$. It is well known that this variety does not depend on the choice of generators. We refer to \cite[\S\S 1,2]{Vog} for further background on this notion.
\end{defn}

\begin{lem}[{See {\it e.g.} \cite[(1.5)(b)]{Vog}}]\label{lem:Nk}
For every $L\in \cM(\fm,K)$ we have $$\mathfrak{S}(L)\subset \cV(\Ann L)\cap \fk^{\bot}\sub \cN(\fg^*)\cap \fk^{\bot}.$$
\end{lem}

Corollary \ref{cor:Yam} follows now from Theorem \ref{thm:Yam}, Proposition \ref{intprop:GeoYam}, and Lemma \ref{lem:Nk}.

\begin{lem}\label{lem:BorBound}
Let $L$ be an $\fm$-module. Suppose that $L$ is finitely generated over a Borel subalgebra $\fb\sub \fm$. Then there exists $N\in \bN$ such that for every irreducible $A\in \cM(\fm,K)$ we have $\dim \Hom_{\fm}(L,A)\leq N$.
\end{lem}
\begin{proof}
By the Casselman embedding theorem \cite[Proposition
8.23]{CaM} there exists a minimal parabolic subgroup $P_0\sub \widetilde{M}$, and a finite-dimensional smooth representation $\sigma$ of $P_0$
such that $A$ may be imbedded into $Ind_{P_0}^{\widetilde{M}}\sigma$. 
Since $A$ is irreducible, $\sigma$ can be chosen to be irreducible, which implies that the unipotent radical of $P_0$ acts trivially.
Thus
 the statement reduces to the case $\widetilde{M}=K$, and $A$ is finite-dimensional. Then $\Hom(L,A)\cong \Hom(A^*,L^*)$. But $A^*$ is generated by its highest weight vector, so this space embeds into $(L^*)^{\fb,\chi}$ for some character $\chi$ of $\fb$. The dimension of this space is in turn bounded by the number of generators of $L$ over $\fb$.
\end{proof}


\subsubsection{Admissible representations}\label{subsubsec:adm}
Let us now define the category of representations to which our main results apply. Let $\widetilde{G}$ be a finite cover of an open subgroup of $G$.
We will call such representations \emph{admissible} and denote the category of such representations by $\cM({\widetilde{G}})$. If $G$ is reductive then we let $\cM(\widetilde{G})\sub \Rep^{\infty}({\widetilde{G}})$ be the subcategory of finitely generated admissible representations (see \cite[\S 11.5]{Wal}).

By Casselman and Wallach (\cite{CasGlob}, \cite[Ch. 11,12]{Wal}) it is equivalent to the category $\cM(\fg,K)$, where $K\sub {\widetilde{G}}$ is a maximal compact subgroup.  The equivalence functor $\cM({\widetilde{G}})\to \cM(\fg,K)$ sends any representation $\pi$ to its space $\pi^{(K)}$ of $K$-finite vectors. This space is known to be dense in $\pi$. 

For a general $G$, fix  a Levi decomposition $G=M\ltimes U$.
Then the cover $\widetilde{G}\to G$ splits over $U$ and thus defines a decomposition $\widetilde{G}=\widetilde{M}\ltimes U$.
\begin{defn}\label{def:adm}
We call a representation $\pi\in \Rep^{\infty}({\widetilde{G}})$   \emph{admissible} if $\pi|_{\widetilde{M}}$ is admissible and $\pi$ is $\cU(\fu)$-finite, i.e. $\Ann_{\cU(\fu)}\pi$ has finite codimension in $\cU(\fu)$.
\end{defn}
\begin{remark}
This definition does not depend on the Levi decomposition.
\end{remark}
}

The following proposition explains the structure of admissible representations of  ${\widetilde{G}}$.
\begin{proposition}\label{prop:AdmStr}
Any admissible $\pi\in \cM({\widetilde{G}})$ admits a finite filtration with associated graded pieces of the form $\pi_i\otimes\chi_i$ where $\pi_i\in \cM(\widetilde{M})$ (considered as a representation of  ${\widetilde{G}}$ using the projection ${\widetilde{G}}\onto \widetilde{M}$) and $\chi_i$ are unitary characters of ${\widetilde{G}}$.
\end{proposition}
\begin{proof}
Since $\pi|_{\widetilde{M}}$ is admissible, it has finite length. Thus $\pi$ also has finite length, and thus it is enough to show that $U$ acts by a unitary character on every irreducible $\pi\in \cM({\widetilde{G}})$.

Since $\Ann_{\cU(\fu)}\pi$ has finite codimension in $\cU(\fu)$, the action of $\fu$ on $\pi$ is locally finite.  Lie's theorem implies now that $\fu$ acts by a character on a non-zero subspace of $\pi$, and thus so does $U$. {Denote this character by $\psi$. We would like to show that $\psi$ is  ${\widetilde{G}}$-invariant. Since $\cU(\fu)/\Ann_{\cU(\fu)}\pi$ is finite-dimensional, it has only finitely many simple modules. Therefore, there are only finitely many $\psi_i$ with non-zero $(\fu,\psi_i)$-eigenspaces in $\pi$.  We conclude that the ${\widetilde{G}}$-orbit of $\psi$ is finite, and thus $\psi$ is fixed by $\fg$. Since $\bfG$ is connected, and the action of $\bfG$ on the space of characters of $\fu$ is algebraic,  we get that $\psi$ is ${\widetilde{G}}$-invariant. Hence so is the $\psi$-eigenspace of $\fu$ in $\pi$.
It is easy to see that this space} is closed, and thus it has to equal $\pi$. We can extend  $\psi$ to ${\widetilde{G}}$ and get $\pi=(\pi\otimes \psi^{-1})\otimes \psi$.
Finally, it is easy to see that all moderate growth characters of unipotent groups are unitary.
\end{proof}

The requirement that $\pi$ is $\cU(\fu)$-finite implies that $\cV(\An \pi)\subset \fu^{\bot}\cong \fm^*$. This, together with the admissibility of $\pi|_M$, implies that $\cV(\An\pi)$ is a finite union of $\bfG$-orbits. Vice versa,  the inclusion $\cV(\An\pi)\subset \fu^{\bot}$ is also necessary to have finitely many orbits, by Proposition \ref{prop:ms}.

For any $\pi\in \cM({\widetilde{G}})$, we define the contragredient representation by $\widetilde{\pi}:=\Sc({\widetilde{G}})\pi^*$.
\begin{lemma}
$\widetilde{\pi}|_{\widetilde{M}}\cong \widetilde{(\pi|_{\widetilde{M}})}$
\end{lemma}
\begin{proof}
We need to show that any $\widetilde{M}$-smooth vector $v\in \pi^*$ is also ${\widetilde{G}}$-smooth.
 \begin{enumerate}[Step 1:]
 \item Proof for the case when $\pi$ is irreducible.\\
 By Proposition \ref{prop:AdmStr} we can write $\pi=\pi_1\otimes \psi$, where $\pi_1\in\cM(\widetilde{M})$  and $\psi$ is a character. Then $v$ is infinitely differentiable under the action of ${\widetilde{G}}$ on the \Fre \, space  $\widetilde{\pi|_{\widetilde{M}}}$. The Dixmier-Malliavin theorem implies  now that $v$ is smooth, i.e. $v\in \widetilde{\pi}=\Sc({\widetilde{G}})\pi^*$.

 \item Proof for the general case.\\
 We prove by induction on the length of $\pi$. Let $\sigma\sub \pi$ be a closed irreducible submodule and let $v_0:=v|_\sigma $. By the previous step we can write $v_0=\sum_i g_i*w_i$ where $g_i\in \Sc({\widetilde{G}})$, $w_i\in \sigma^*$ and the sum is finite. By the Hahn-Banach theorem there exist $w_i'\in \pi^*$ s.t. $w_i'|_{\sigma}=w_i$. Let $w:=v-\sum_i g_i*w'_i$. It remains to show that $w$ is ${\widetilde{G}}$-smooth. It is easy to see that $w|_{\sigma}=0$. Hence the assertion follows from the induction assumption applied to $\pi/\sigma$.
 \end{enumerate}

\end{proof}

The lemma implies that  $\widetilde{\widetilde{\pi}}\cong \pi$.
From this and Lemma \ref{lem:Frob} we obtain the following corollary.
\begin{cor}\label{cor:Frob}
For any \DimaI{semi-}algebraic subgroup $H\subset {\widetilde{G}}$ and any $\tau \in \cM(H)$ and  $\pi\in \cM({\widetilde{G}})$ we have
$$\Hom_{\widetilde{G}}(\ind_H^{\widetilde{G}}\tau,\pi)\cong
\Hom_{{\widetilde{G}}\times H}(\Sc({\widetilde{G}}),\pi\otimes
 \widetilde{\tau\delta_H}),$$
where ${\widetilde{G}}\times H$ acts on ${\widetilde{G}}$ by left and right shifts.
\end{cor}
\begin{proof}
Since $\Sc({\widetilde{G}})\ind_H^{\widetilde{G}}\tau=\ind_H^{\widetilde{G}}\tau$, we have
$\Hom_{\widetilde{G}}(\ind_H^{\widetilde{G}}\tau,\pi)= \Hom_{\widetilde{G}}(\ind_H^{\widetilde{G}}\tau,\widetilde{\pi}^*)$. By Lemma \ref{lem:Frob} we have  $$ \Hom_{\widetilde{G}}(\ind_H^{\widetilde{G}}\tau,\pi)= \Hom_{\widetilde{G}}(\ind_H^{\widetilde{G}}\tau,\widetilde{\pi}^*)\cong \Hom_H(\tau,\widetilde{\pi}^*|_H\delta_H^{-1}\delta_{\widetilde{G}})\cong \Hom_H(\widetilde{\pi}|_H,\tau^*\delta_H^{-1}\delta_{\widetilde{G}})$$
Now we note that $\Sc(G)\cong \ind_{\Delta H}^{{\widetilde{G}}\times H}\C$ and use Lemma \ref{lem:Frob} again.
\begin{multline*}
\Hom(\widetilde{\pi}|_H,\tau^*\delta_H^{-1}\delta_{\widetilde{G}})\cong \Hom_{\Delta H}(\C,\widetilde{\pi}^*\otimes\tau^*\delta_H^{-1}\delta_{\widetilde{G}})\cong
\Hom_{{\widetilde{G}}\times H}(\Sc({\widetilde{G}}),\widetilde{\pi}^*\otimes \tau^*\delta_H^{-1})\cong \\ \cong \Hom_{{\widetilde{G}}\times H}(\Sc({\widetilde{G}}),\pi\otimes \widetilde{\tau\delta_H})
\end{multline*}
\end{proof}


\begin{defn}
For any $\Pi,\tau\in \Rep^{\infty}({\widetilde{G}})$, we define the multiplicity of $\tau$ in $\Pi$ as $$m(\Pi,\tau):=\dim \Hom_{\widetilde{G}}(\Pi,\tau).$$ We say that $\Pi$ \emph{has finite multiplicities} if $m(\Pi,\tau)$ is finite for every $\tau\in \cM({\widetilde{G}})$. \end{defn}

For any $\fg$-module $\pi$, denote by $\AnV(\pi)\subset \fg^*$ the associated variety of the annihilator of $M$ in $\cU(\fg)$, {\it i.e.} $\AnV(\pi)=\cV(\An(\pi))$.

If $\pi \in \cM({\widetilde{G}})$, $\AnV(\pi)$ is a union of nilpotent orbits in $\fm^*$.
For any closed conical $\bfG$-invariant subset $\Xi\subset \cN(\fm^*)\sub\fg^* $, denote by $\cM_{\Xi}({\widetilde{G}})\sub \cM({\widetilde{G}})$ the subcategory consisting of representations $\pi$ with $\AnV(\pi)\subset \Xi$.

\DimaI{
Let us now define a similar notion of Harish-Chandra modules.
\begin{defn}\label{def:adm}
Let $K\sub \widetilde{G}$ be a maximal compact Lie subgroup. 
A \emph{Harish-Chandra} $(\g,K)$-module $L$ is a $\fg$-module $L$ with an action of $K$ that forms a Harish-Chandra $(\fm,K)$-module, and is $\cU(\fu)$-finite.
We will denote the category of \emph{Harish-Chandra} $(\g,K)$-modules by $\cM(\fg,K)$. For any closed conical $\bfG$-invariant subset $\Xi\subset \cN(\fm^*)\sub\fg^* $, denote by $\cM_{\Xi}(\g,K)\sub \cM(\g,K)$ the subcategory consisting of modules $L$ with $\cV(\An L)\subset \Xi$.
\end{defn}

\begin{rem}
A general Harish-Chandra $(\g,K)$-module $L$ may not be the module of $K$-finite vectors of an admissible $\pi\in \cM(\widetilde{G})$. A necessary and sufficient condition is that all the eigencharacters of $\fu$ on $L$ have imaginary values on the real points of $\fu$. The reason is that a smooth character of $U$ has moderate growth if and only if it is unitary.
\end{rem}
}


\subsection{Main results}
Fix a Levi decomposition $\bf G=MU$, and let $\bf X$ be an algebraic $\bf G$-manifold.

The natural projection $p:\fg\onto \fm$ defines an embedding $i:\fm^*\into \fg^*$. Let $\fz(\fm)$ denote the center of $\cU(\fm)$.


\begin{theorem}\label{thm:fin}
Let  $I\subset \cU(\fg)$ be a two-sided ideal such that $I\cap \fz(\fm)$ is cofinite in $\fz(\fm)$, and
$I\cap \cU(\fu)$ is cofinite in $\cU(\fu)$.
Assume that $\bf X$ is $\cV(I)$-spherical, and let $\cE$ be a twisted algebraic ${\widetilde{G}}$-equivariant vector bundle on $X$. Then $\Sc( X,\cE)_I\in \cM(G)$.
\end{theorem}
\begin{proof}
We have $\Sc( X,\cE)\in \Rep^{\infty}(G)$, and thus $\Sc( X,\cE)_I\in\Rep^{\infty}(G)$.
Thus, by the assumptions on $I$, it is enough to show that $\Sc( X,\cE)_I\in \cM(M)$.

Let $K\subset M$ denote a maximal  compact subgroup, and let $\bf K$ be the corresponding subgroup of $\bf M$.
Since $\Sc({ X},\cE)_I$ is $\fz(\fm)$-finite, by \cite[Corollary 5.8]{CasGlob} it is enough to show that it has finite multiplicities as a representation of $K$. Let $\cE'$ denote the tensor product of the dual bundle to $\cE$ with the bundle of densities on $X$. Then $(\Sc({ X},\cE)_I)^*\cong \Sc^*({ X},\cE')^I$.
Fix a $K$-type $\rho\in \widehat{K}$.

We have  $\Hom_K(\Sc({ X},\cE)_I,\rho^*)\into \Hom_K(\rho, \Sc^*({ X},\cE')^I)$, and by Lemma \ref{lem:Frob} we have
$$\Hom_K(\rho, \Sc^*({ X},\cE')^I)\cong \Hom_G(\delta_G\ind_{K}^G\rho, \Sc^*({ X},\cE')^I)$$
Let $\cE''$ be the twisted algebraic $G$-equivariant vector bundle on $G/K$ such that $\delta_G\ind_{K}^G\rho\cong \Sc(G/K,\cE'')$. By Lemmas \ref{lem:NF} and \ref{lem:SKT},
$$\Hom_G(\delta_G\ind_{K}^G\rho,\Sc^*({ X},\cE)^{I})\cong \Sc^*( G/K \times X,\cE''\boxtimes \cE')^{I,\Delta G}$$

Since $I\cap \cU(\fu)$ is cofinite in $\cU(\fu)$, we have $\cV(I)\subset \fm^*$.
Since $I\cap \fz(\fm)$ is cofinite in $\fz(\fm)$, we further have  $\cV(I)\subset \cN(\fm^*)$.
By Corollary \ref{cor:AbsSpher}, we obtain that $\bf M/K$ is a $\cV(I)$-spherical  $\bf M$-manifold, which by Lemma  \ref{lem:indSpher} implies that $\bf G/K$ is a $\cV(I)$-spherical  $\bf G$-manifold.
By Theorem \ref{Thm:Dist}, the space $\Sc^*( G/K\times X,\cE''\boxtimes \cE')^{I,\Delta G}$ is finite-dimensional.
\end{proof}
It is easy to see that the condition that $I\cap \fz(\fm)$ is cofinite in $\fz(\fm)$ does not depend on the choice of the Levi decomposition.

\begin{cor}\label{cor:fin}
Let $\Xi\sub \cN(\fm^*)$ be a closed conical subset.
Assume that $\bf X$ is $\Xi$-spherical, and let $\cE$ be a twisted algebraic $G$-equivariant vector bundle on $X$.
 Then any $\pi\in \cM_{\Xi}(G)$ has (at most) finite multiplicity in $\Sc(X,\cE\otimes \sigma)$, i.e. $\dim \Hom_G(\Sc(X,\cE),\pi)<\infty$.
\end{cor}
The corollary follows from Theorem \ref{thm:fin} by taking $I:=\Ann(\pi)$.

\begin{remark}
Corollary \ref{cor:fin} stays true under a weaker assumption on $\pi$: it works for any $\fg$-module $\pi$, such that $\cV(\pi)\sub \Xi$, and the action of $\fm$ on $\pi$ integrates to an admissible representation of $M$. Indeed, the argument in the proof of Theorem \ref{thm:fin} shows that the assumption $\cV(\pi)\sub \Xi$  implies that $\Sc(X,\cE)_{\Ann \pi}\in \cM(M)$.
\end{remark}

\begin{cor}\label{cor:AdmSpher}
If $\bfG$ is reductive and $\bfX$ is spherical then for every ideal $J\subset \fz(\cU(\fg))$ of finite codimension, $\Sc(X,\cE)_{J\cU(\fg)}\in \cM(G)$.
\end{cor}


\begin{cor}\label{cor:FinMult} Assume that $G$ is reductive and let $\Pi\in \Rep^{\infty}(G)$. Then the following are equivalent
\begin{enumerate}[(i)]
\item $\Pi$ has finite multiplicities \label{it:FinMult}
\item For any ideal $J\subset \fz(\cU(\fg))$ of finite codimension, $\Pi_{J\cU(\fg)}\in \cM(G)$. \label{it:AdmQuot}
\end{enumerate}
\end{cor}
\begin{proof}
\eqref{it:AdmQuot} $\Rightarrow$ \eqref{it:FinMult} is obvious.
For the implication \eqref{it:FinMult}$\Rightarrow$ \eqref{it:AdmQuot} denote $I:=J\cU(\fg)$. By \cite[Corollary 5.8]{CasGlob} it is enough to show that any $K$-type has finite multiplicities in $\Pi_I$, since $\Pi_{I}$ is $\fz$-finite by construction. Let $\sigma \in \widehat{K}$. Then by Lemma \ref{lem:Frob2},
$$\sigma\otimes_K \Pi_{I}\cong \ind_K^G\sigma \otimes_G \Pi_{I}\cong (\ind_K^G\sigma)_{I} \otimes_G \Pi,$$
and to show that this is finite-dimensional it is enough to show that $(\ind_K^G\sigma)_I$ has finite length (since $\Pi$ has finite multiplicities). To show this it is again enough to show that every $K$-type has finite multiplicities in it.
Using the compactness of $K$ and Lemma \ref{lem:Frob2} we have
$$
\Hom_K(\rho,(\ind_K^G\sigma)_{I})\cong
(\rho^*\otimes(\ind_K^G\sigma)_{I})_{K}\cong
(\ind_K^G\rho^*\otimes(\ind_K^G\sigma)_{I})_G$$
Since $I$ is central, $I\times I$ acts on $(\ind_K^G\rho^*\otimes\ind_K^G\sigma)_G$, and   we have
$$(\ind_K^G\rho^*\otimes(\ind_K^G\sigma)_{I})_G\cong ((\ind_K^G\rho^*\otimes\ind_K^G\sigma)_G)_{I\times I} $$
Since $K$ is compact, we have
$$(\ind_K^G\rho^*\otimes\ind_K^G\sigma)_G\cong \Hom_{K\times K}(\rho\otimes \sigma^*,\Sc(G\times G/\Delta G))$$
Altogether, we have $\Hom_K(\rho,(\ind_K^G\sigma)_{I})\cong\Hom_{K\times K}(\rho\otimes \sigma^*,\Sc(G\times G/\Delta G)_{I\times I}),$
which is finite-dimensional since $\Sc(G\times G/\Delta G)_{I\times I}\in \cM(G\times G)$ by Corollary \ref{cor:AdmSpher}.
\end{proof}

%

\subsection{Applications to branching problems}\label{subsec:RepBranch}
Let $\bf H\sub G$ be an algebraic subgroup defined over $\R$.
Let $\bf L$ denote the quotient of $\bfH$ by its unipotent radical,
and let $p_{\bf L}:\bfH \to {\bf L}$ denote the projection.
\DimaH{Let $\widetilde{G}$ be a finite cover of an open subgroup of $G$, and let $\widetilde{H}$ be an open subgroup of the preimage of the real points of $\bf H$ in $\widetilde{G}$.}



\begin{prop}\label{prop:branch}
Let $\bfO_1\subset \fm^*$ and $\bfO_2\subset \fl^*$ be nilpotent orbits. Suppose that one of the following holds:
\begin{enumerate}[(a)]
\item \label{it:DSpher} $\bf G\times H / \Delta H$ is $\overline{\bfO_1}\times \overline{\bfO_2}$-spherical

\item \label{it:mO12} $\dim \bfO'_1\cap p_{\fh}^{-1}(\bfO'_2)\leq (\dim \bfO'_1+\dim\bfO'_2)/2$ for any $\bfO'_1\subset \overline{\bfO_1}$ and $\bfO'_2\subset \overline{\bfO_2}$.
\item \label{it:mO1}  $\bfO_1={\bf O_P}$ for some parabolic subgroup $\bf P\subset M$\DimaB{, $\bf M/P$ is an $\overline{\bfO_2}$-spherical $\bfH$-space, and either $\bfP$ is a Borel subgroup of $\bf M$ or $[\fm,\fm]$ is a product of Lie algebras of type A.}
\item \label{it:mO2} $\bfO_2=\bfO_\bfQ$ for some parabolic subgroup $\bf Q\subset L$\DimaB{, $\bf G/p_{\bf L}^{-1}(Q)$ is an $\overline{\bfO_1}$-spherical $\bf G$-space, and either $\bf Q$ is a Borel subgroup of $\bf L$ or $[\fl,\fl]$ is a product of Lie algebras of type A.}
\item \label{it:mRich} \DimaB{ $\bfO_1=\bf O_P$ for some parabolic subgroup $\bf P\subset G$,  $\bfO_2=\bf O_Q$ for some parabolic subgroup $\bf Q\subset H$, and the set of double cosets $\bf Q\backslash G/P$ is finite. }
\end{enumerate}

Then for every $\pi\in \cM_{\overline{\bfO_{1}}}({\widetilde{G}})$ and $\tau\in \cM_{\overline{\bfO_2}}({\widetilde{H}})$, we have $$\dim \Hom_{\widetilde{H}}(\pi|_{\widetilde{H}},\tau)< \infty \text{ and } \dim( (\pi|_{\widetilde{H}})\otimes_{\fh}\tau)< \infty $$
Moreover, \DimaH{there exist Zariski open dense subsets $U\sub \cK_M$ and $V\sub \cK_L$ such that for every $K\in U$ and $K'\in V$, and every finite cover $\widetilde{K}$ of $K$ and $\widetilde{K'}$\ of $K'$, every $A\in \cM_{\overline{\bfO_{1}}}(\fg,\widetilde{K})$
and $B\in \cM_{\overline{\bfO_{2}}}(\fh,\widetilde{K'})$, and every $i\geq 0$   we have $\dim \Tor^i_{\fh}(A|_{\fh},B)<\infty$. }
\end{prop}


\begin{proof}
\DimaC{By \DimaE{Corollary \ref{cor:invO} and Lemma} \ref{lem:NonRed}, conditions \eqref{it:DSpher} and \eqref{it:mO12} are equivalent. By Corollary \ref{cor:RichBran}, \eqref{it:mRich} implies \eqref{it:DSpher}. By Corollary \ref{cor:BranchSpher}, conditions \eqref{it:mO1} and \eqref{it:mO2} also imply condition  \eqref{it:DSpher}, since in Lie algebras of type A all nilpotent orbits are Richardson.}
We will thus assume that \eqref{it:DSpher} holds.

\DimaH{
By Proposition \ref{intprop:GeoYam} this implies that there exist Zariski open dense subsets $U\sub \cK_G$ and $V\sub \cK_H$ such that for every $K\in U$ and $K'\in V$, 
$$\overline{\bfO_1}\times \overline{\bfO_2}\cap \Delta \fh^{\bot}\cap (\fk\times \fk')^{\bot}=\{0\}$$
By Theorem \ref{thm:Yam} (which is proven in \cite{Yam} without the assumption that $\fg$ is reductive)
this implies that every finite covers $\widetilde{K}$ of an open subgroup of $K$, and $\widetilde{K'}$ of an open subgroup of $K'$, every $A\in \cM_{\overline{\bfO_{1}}}(\fg,\widetilde{K})$
and $B\in \cM_{\overline{\bfO_{2}}}(\fh,\widetilde{K'})$, $A\otimes_\C B$ is finitely generated over $\Delta \fh$.

This implies that for any $i\geq 0$ we have 
$$\dim \Tor^i_{\Delta \fh} (A\otimes_{\C} B,\C)<\infty$$
and thus $\dim \Tor^i_{\fh}(A|_{\fh},B)<\infty$.

Now let $\pi\in \cM_{\overline{\bfO_{1}}}({\widetilde{G}})$ and $\tau\in \cM_{\overline{\bfO_2}}({\widetilde{H}})$. Taking $A=\pi^{(K)}$ and $B=\tau^{(K')}$ we obtain 
$ \dim ((\pi|_{\widetilde{H}})\otimes_{\fh}\tau)< \infty $. 
Now 
$$\Hom_{\widetilde{H}}(\pi|_{\widetilde{H}},\tau)\cong \Hom_{\widetilde{H}}(\pi|_{\widetilde{H}},\widetilde{\tau}^*)\cong (\pi\otimes_{\C}\widetilde{\tau}^*)^{\widetilde{H}}\into  ((\pi|_{\widetilde{H}})\otimes_{\fh}\widetilde{\tau})^*$$
 and thus $\dim \Hom_{\widetilde{H}}(\pi|_{\widetilde{H}},\tau) \leq \dim ((\pi|_{\widetilde{H}})\otimes_{\fh}\widetilde{\tau})< \infty $. }

\end{proof}





\begin{cor}
Let $\Xi\subset \cN(\fm^*)$ be a closed $\bfG$-invariant subset,  and suppose that $\bf G/H$ is $\Xi$-spherical. Then for any $\pi\in \cM_{\Xi}({\widetilde{G}})$ and any finite-dimensional $\tau\in \cM({\widetilde{H}})$ we have $$\dim \Hom_{\widetilde{H}}(\pi|_{\widetilde{H}},\tau)< \infty$$
\end{cor}

\begin{cor}\label{cor:IndSpher}
Suppose that $\bf G$ is reductive, and that the unipotent radical $\bf V$ of $\bf H$ equals the unipotent radical of a parabolic subgroup $\bf R\sub G$. Let $\bf Q \sub L$ be a parabolic subgroup, and suppose that $\bf Q$ is a spherical subgroup of $\bf R/V$.
Then for every $\tau\in  \cM_{\overline{\bfO_{\bfQ}}}({\widetilde{H}})$, the induction $\ind_{\widetilde{H}}^{\widetilde{G}}\tau$ has finite multiplicities.

\end{cor}
\DimaC{
\begin{proof}
By the assumption $\bf Q$ is a spherical subgroup of $\bf R/V$.
This implies that the parabolically induced subgroup $\bf QV \sub G$ is also spherical. Thus that the condition \eqref{it:mRich} of Proposition \ref{prop:branch} is satisfied for any Borel subgroup $\bf P\sub G$ . Thus for every $\pi\in \cM({\widetilde{G}}) $ we have $\dim \Hom_{\widetilde{H}}(\pi|_{\widetilde{H}},\tau)< \infty $. The corollary follows now from Corollary \ref{cor:Frob} on Frobenius reciprocity.
\end{proof}

\begin{proof}[Proof of Corollary \ref{IntCor:branch2}]
In this corollary both $\bf G$ and $\bf H$ are reductive.
Thus part \eqref{int:FinMultInd} follows from the previous corollary by taking $\bf R=G$.
\DimaH{
By Corollary \ref{cor:RichBran}, the condition in part \eqref{int:FinMultRes}
implies that $\bf G/B_H$ is an $\overline{\bf O_P}$- spherical $\bf G$-variety, which is the condition in part \eqref{int:BoundMultResGen} for $\bf O=O_P$. Thus it is enough to prove part \eqref{int:BoundMultResGen}, {\it i.e.} that if $\bf G/B_H$ is an $\overline{\bfO}$- spherical $\bf G$-variety then 
  for every $\pi\in \cM_{\overline{\bfO}}({\widetilde{G}})$, the restriction $\pi|_{\widetilde{H}}$ has multiplicities.
By Corollary \ref{cor:Yam} applied to $\fb_{\fh}\sub \fg$, there exists a Zariski open dense $U\sub \cK_G$  such that for every $K \in U$, $\pi^{(K)}$ is finitely generated over $\fb_{\fh}$. By Lemma \ref{lem:BorBound}, this implies that $\pi|_{\widetilde{H}}$ has \DimaH{bounded} multiplicities.}
\end{proof}
}

\begin{proof}[Proof of Example \ref{Ex:HShal}]
Let $\bf G$ be $\GL_{2n}$,  $\bf L'\subset G$ be the Levi subgroup $\GL_n\times\GL_n$, $\bf R=L' V$ be the corresponding standard parabolic,
$\bf L:=\Delta GL_n \sub L'$, and $\bf H=LV$.
Let $\bf Q\sub L$ be the mirabolic subgroup. Then it is easy to see that $\bf Q$ is a spherical subgroup of $\bf L'$ (this is also shown in \cite{MWZ_GL} and \cite{Stem}). By Corollary \ref{cor:IndSpher} this implies that for every $\tau\in  \cM_{\overline{\bfO_{\bfQ}}}({\widetilde{H}})$, the induction $\ind_{\widetilde{H}}^G\tau$ has finite multiplicities.
\end{proof}


\appendix


\section{Proof of Proposition \ref{prop:ms}}\label{app:Pfms}
Let $\bfG$ be an algebraic group, $\bf U$ be its unipotent radical, and $\bf M=G/U$.

For the proof we will need  the following lemma.

\begin{lemma}\label{lem:non-nilp} Let $V$ be an $\fm$-module, and let $v\in V$ be a non-zero vector such that there exists a semi-simple $s\in \fm$ with $sv=v$. Then there exists a non-nilpotent $t\in \fm^*$ that is orthogonal to the stabilizer $\fm_v$ of $v$.
\end{lemma}

\begin{proof}
Let $\fl:=\fm^s\subset \fm$ denote the centralizer of $s$.
We have a decomposition $\fm=\fl\oplus \fr$, where $\fr$ is the direct sum of all eigenspaces of $ad(s)$ on $\fm$ corresponding to non-zero eigenvalues.
This defines a dual decomposition $\fm^*=\fl^*\oplus \fr^*$, with $\fl^*$ orthogonal to $\fr$. Under this decomposition, an element $t\in \fl^*$ is nilpotent as an element of $\fl^*$ if and only if it is nilpotent as an element of $\fm^*$. One can see this using an identification $\fm\cong \fm^*$ given by a non-degenerate invariant quadratic form.

Since the adjoint action of $s$ preserves $\fm_v$, we have a  decomposition $\fm_v=\fl_v\oplus \fr_v$.
Thus it is enough to find a non-nilpotent element $t\in \fl^*$ orthogonal to $\fl_v$.

Let $\fc$ denote the center of $\fl$.
 We will prove more: there exists $t\in \fl_v^{\bot}$ that does not lie in $\fc^{\bot}$. Suppose the contrary: $\fl_v^{\bot}\subset \fc^{\bot}$. But then $\fc\sub \fl_v$, which is not true since $s\in \fc$ and $s\notin \fl_v$.
\end{proof}


\begin{proof}[Proof of Proposition \ref{prop:ms}]
We will prove  the proposition by induction on the depth of $\fu$.
Let $\mathfrak{e}$ be the center of $\fu$.
By the induction hypothesis, it is enough to prove that  $\Xi\sub\ \mathfrak{e}^{\bot}$.
Assume the contrary, and let $\bfO\subset \Xi$ be an open orbit s.t. $\bfO\not\subset \mathfrak{e}^{\bot}$.
Let $p:\fg^*\to \mathfrak{e}^*$ denote the restriction.
We will consider the Lie algebra action of $\fg$ on $\fg^*$.
For $\alpha \in \fg$ and $\nu \in \fg^*$
we will denote the result of this action by $\alpha \cdot \nu$. Since $\bf U$ acts trivially on $\mathfrak{e}$ and on $\mathfrak{e}^*$,  the action of $\bfG$ on $\mathfrak{e}^*$ factors through $\bf M$.
\begin{enumerate}[Step 1.]
\item For any $y\in p(\bfO)$ there exists  $\mu_y\in \fm$ s.t. $\mu_y \cdot y=y$.

Let $x\in p^{-1}(y)\cap \bfO$, and let $\mathbf{\Lambda}:=\Span\{x\}$. Since $\Xi$ is conical and $\bfO$ is open in $\Xi$, $\mathbf{\Lambda} \cap \bfO$ is open in $\mathbf{\Lambda}$. Thus $x$ lies in the tangent space to $\bfO$ at $x$. Thus there exists some $\alp\in \fg$ s.t. $\alp \cdot x=x$. Since $p$ is $\bfG$-equivariant, we have $\alp\cdot y=y$, and thus $\mu_y\cdot y =y$, where $\mu_y$ is the projection of $\alp$ to $\fm\cong \fg/\fu$.

\item For any $y\in p(\bfO)$ there is a semi-simple $s_y\in \fm$ s.t. $s_y \cdot y=y$. By semi-simple we mean that $s_y$ acts semi-simply on any algebraic representation of $\bf M$.

Let $\mu_y$ be as in the previous step, and let $\mu_y = s_y + n_y$ be the Jordan decomposition of $\mu_y$. Since $n_y$ acts nilpotently on $\mathfrak{e}^*$, and the actions of $s_y$ and $n_y$ commute, we have $s_y \cdot y = y$.

\item 
For any $y\in p(\bfO)$, there exists a non-nilpotent
$t_y\in \fm^*$ such that
$\bfO\cap p^{-1}(y)$ is locally invariant w.r.t. shifts in  $t_y$, {\it i.e.} for any $x\in\bfO\cap p^{-1}(y),$ we have $t_{y}\in T_{x}(\bfO\cap p^{-1}(y))$.

Fix $y\in p(\bfO)$.  By  the previous step and Lemma \ref{lem:non-nilp} there exists a non-nilpotent $t_y\in \fm^*$ orthogonal to the stabilizer of $y$ in $\fm$.
Let $x\in\bfO\cap p^{-1}(y)$. It is enough to show that $\mathfrak{e}\cdot x \ni t_y$. Consider the map $\phi_x: \mathfrak{e} \to \fg^*$ defined by $\phi(\alp):=\alp\cdot x$. It is easy to see that $\Im(\phi)\sub \fm^*$. Thus we will consider $\phi$ as a map $\mathfrak{e}\to \fm^*$. Let $\psi=\phi^t:\fm\to \mathfrak{e}^*$. Then $\Im \phi=(\Ker\psi)^{\bot}$. It is easy to see that $\psi(\alp)=\alp\cdot y$. Thus $\Ker \psi = \fm_y$. Thus, $t_{y}\in  (\Ker\psi)^{\bot}=\Im \phi=\mathfrak{e}\cdot x$.

\item For any $y\in p(\bfO)$, there exists non-nilpotent $t_y\in \fm^*$ such that
 $\overline\bfO\cap p^{-1}(y)$ includes an affine line in the direction of $t_y$ through any point $x\in \bfO\cap p^{-1}(y)$.

We take  $t_y$ as in the previous case. Note that $\bfO\cap p^{-1}(y)$ is an orbit of $\bfG_y$ and thus  is smooth.
Let $N:=(\bfO\cap p^{-1}(y))(\C)$. We can consider $t_y$ as a constant vector filed on $N$. Let $x\in N$. The existence and uniqueness theorem for solutions of ODEs imply that $x+\C t_y\cap N$ is open in $x+\C t_y$ and thus it is Zariski dense in $x+\C t_y$. Let $L\subset p^{-1}(y)$ be the affine line through $x$ in direction $t_y$. We get that $L\cap \overline\bfO\cap p^{-1}(y)$ is Zariski dense in $L$. But it is also closed there.
Therefore  $L\cap \overline\bfO\cap p^{-1}(y)=L$.
\item There exists non-nilpotent
$t_0\in \overline\bfO\cap \fm^*$.

Take $y\in p(\bfO)$ and  $x\in \bfO\cap p^{-1}(y)$.
Let $\mathbf{\Lambda}:=\Span\{x\}$ and
$\mathbf{\Pi}:=\Span\{x,t_y\}$.
Since $\Xi$ is conical and $\bfO$ is open in $\Xi$,  $\mathbf{\Lambda}\cap \bfO$ is dense in $\mathbf{\Lambda}$. Together with the previous step, this implies that $\mathbf{\Pi}\cap\overline{\bfO}$ is dense in $\bf \Pi$. Therefore $\mathbf{\Pi}\sub \overline\bfO$, and thus $t_y\in \overline\bfO\cap \fm^*$.
\item Contradiction.

The fact that $\overline\bfO \cap \fm^*$ is conical and has finitely many orbits implies that it lies inside the nilpotent cone of $\fm^*$. This contradicts the previous step.
\end{enumerate}

\end{proof}

\section{Example of strict inequality in Definition \ref{def:Ospher} - by Ido Karshon}\label{app:Ido}

Let $W$ be a symplectic vector space of dimension $2n$.
Let $\bfG:=\Sp(W)\times \Sp(W\oplus W)$ and
$$\bfH:= \{(Y,\begin{pmatrix}Y & 0 \\
0 & Y \\
\end{pmatrix})\}\subset \bfG$$
Then
$$\fh^{\bot}:= \{(-(A+C),\begin{pmatrix}A & B \\
-B^* & C \\
\end{pmatrix})\, \vert\, A=-A^*,C=-C^*\}\subset \fg,$$
where $A^*,B^*,$ and $C^*$ denote the conjugate operators w.r. to the symplectic form.
Let $\bfO\subset \cN(\fg^*)$ be the product of minimal orbits of $\mathfrak{sp}^*(W)$ and $\mathfrak{sp}^*(W\oplus W)$.
We identify $\fg$ with $\fg^*$ using the Killing form, and
recall that the minimal orbit in the symplectic Lie algebra consists of rank one operators of the form $vv^*$, where $v$ is a vector. Thus
 $$\fh^{\bot}\cap \bfO= \{(-(aa^*+bb^*),\begin{pmatrix}aa^* & ab^* \\
ba^* & bb^* \\
\end{pmatrix})\,\vert\, a,b \in W \text{ are colinear} \}\subset \fg$$
Thus $\dim \fh^{\bot}\cap \bfO=\dim W +1$, while $\dim \bfO=3\dim W$.
Thus for $\dim W>2$ we have $\dim \fh^{\bot}\cap \bfO<\dim \bfO/2$, and thus, \DimaC{by \DimaE{Corollary \ref{cor:invO}}, $c_{\bf O}^{-1}({\bf G/H})<0.$}

%
%
%

\end{document}